\documentclass[11pt]{amsart}
\usepackage{amsmath}
\usepackage{amssymb}
\usepackage{amscd}
\def\NZQ{\mathbb}               % the font for N,Z,Q,R,C
\def\NN{{\NZQ N}}

\def\ZZ{{\NZQ Z}}
\def\RR{{\NZQ R}}

\newtheorem{Theorem}{Theorem}[section]
\newtheorem{Lemma}[Theorem]{Lemma}
\newtheorem{Corollary}[Theorem]{Corollary}
\newtheorem{Proposition}[Theorem]{Proposition}

\newtheorem{Question}[Theorem]{Question}
\let\epsilon\varepsilon
\let\phi=\varphi
\let\kappa=\varkappa

\begin{document}
\title{Asymptotic Multiplicities}
\author{Steven Dale Cutkosky }
\thanks{Partially supported by NSF}

\address{Steven Dale Cutkosky, Department of Mathematics,
University of Missouri, Columbia, MO 65211, USA}
\email{cutkoskys@missouri.edu}

%\begin{abstract} 
%\end{abstract}

\maketitle

\section{Introduction} 

This paper is based on a talk  on graded families of ideals and filtrations and associated asymptotic multiplicities, which was given by the author at the Third International Symposium on Groups, Algebras and Related Topics, celebrating the 50th anniversary of the Journal of Algebra, held in Beijing, China, from June 10 - 16, 2013.

Section \ref{Section2} of this paper discusses these concepts. The main consideration is the behavior of the length
$\ell_R(R/I_n)$ on a graded family of $m_R$-primary ideals  in a $d$-dimensional local ring $R$. The classical result is for the case when
$I_n=I^n$ are powers of a fixed ideal. In this case it is classical that 
$\ell_R(R/I_n)$ is a polynomial of degree $d$ for large $n$. This is however not the case for general graded families. In fact, the associated graded ring $\bigoplus_{n\ge 0}I_n$ is in general not a finitely generated $R$-algebra, so we cannot
expect  $\ell_R(R/I_n)$ to  be polynomial like for large $n$ in general. The remarkable thing is that under very general conditions this length does asymptotically approach a polynomial of degree $d$; that is, the limit
\begin{equation}\label{eqI1}
\lim_{n\rightarrow\infty}\frac{\ell_R(R/I_n)}{n^d}\in\RR
\end{equation}
exists. 
The main result recalled in this section is Theorem \ref{Theorem1}, which is  Theorem 5.3 of \cite{C6}. It shows that the limit (\ref{eqI1}) exists for all graded families of $m_R$-primary ideals if and only if the completion of $R$ is generically reduced. 
A condensed proof of Theorem \ref{Theorem1} is given in Section \ref{Section5}, referring to results from our papers \cite{C4} and \cite{C6}. We will build on this proof to establish some of the applications in the next section of this paper. 

In Section \ref{Section5}, we give applications of Theorem \ref{Theorem1} and the method of its proof. Some of the results are quoted from \cite{C4} and \cite{C6}. Most of the results are new to this paper. We give complete proofs for the new theorems. They include a Minkowski type formula for graded families of ideals (Theorem \ref{TheoremM7}),  some formulas for 
limits associated to divisorial ideals, including a proof that the local volume exists as a limit under very general conditions (Corollary \ref{CorollaryM40}), and a proof that  epsilon multiplicity exists  as a limit for modules under very general conditions (Theorem \ref{TheoremC}).
The longer proofs of Theorem \ref{TheoremM7} and Theorem \ref{TheoremC} are given in Sections \ref{Section6} and \ref{Section7}. 

In Section \ref{Section4}, we consider  other polynomial like properties that $\ell_R(R/I_n)$ could asymptotically have. 
The basic conclusion is that  Theorem \ref{Theorem1} is the strongest statement of polynomial like behavior that is always true.

In the course of the paper, we discuss in detail previous results and some of the history of the problems.

\section{Asymptotic multiplicities}\label{Section2} 

\subsection{Multiplicity, graded families and filtrations of ideals}
Let $(R,m_R)$ be a (Noetherian) local ring of dimension $d$. 

A family of ideals $\{I_n\}_{n\in\NN}$ of $R$ is called a graded family of ideals if $I_0=R$ and $I_mI_n\subset I_{m+n}$ for all $m,n$.
$\{I_n\}$ is a filtration of $R$ if we further have that $I_{n+1}\subset I_n$  for all $n$. 

The most basic example is $I_n=J^n$ where $J$ is a fixed ideal of $R$.

Suppose that $N$ is a finitely generated $R$-module, and $I$ is an $m_R$-primary ideal.
Let
$$
t=\dim R/\mbox{ann}(N)
$$ 
be the dimension of $N$. Let $\ell_R$ be the length of an $R$-module.

The theorem of Hilbert-Samuel is that the function $\ell_R(N/I^nN)$ is a polynomial in $n$ of degree $t$ for $n\gg0$
(Chapter VIII, \cite{ZS2}). This polynomial is called the Hilbert-Samuel polynomial.

The multiplicity of a finitely generated $R$-module $N$ with respect to an $m_R$-primary ideal $I$ is the leading coefficient of this polynomial times $t!$; that is, 
\begin{equation}\label{eq3}
e_I(N)=\lim_{n\rightarrow \infty}\frac{\ell_R(N/I^nN)}{n^t/t!}.
\end{equation}
This multiplicity is always a natural number.
We write $e(I)=e_I(R)$.

The following example  shows that irrational limits occur for natural filtrations.

\begin{Theorem}(Example 6 \cite{CS}) There exists an inclusion $R\rightarrow S$ of $d$-dimensional normal domains, essentially of finite type over the complex numbers, such that 
$$
\lim_{n\rightarrow\infty}\frac{\ell_R(R/I_n)}{n^d}
$$
exists but is an irrational number, where $I_n=m_S^n\cap R$.
\end{Theorem}

The irrationality of the limit implies that $\bigoplus_{n\ge 0}I_n$ is not a finitely generated $R$-algebra.

\subsection{Limits of multiplicities of graded families of ideals}

The following theorem makes use of the Minkowski inequality of ideals in local rings by Teissier \cite{T} and Rees and 
Sharp \cite{RS} (c.f. Section 17.7 \cite{SH}). The interesting conclusion in 1) has  already been pointed out by Ein, Lazarsfeld and Smith \cite{ELS}
and by Musta\c{t}\u{a} \cite{Mus}.

\begin{Theorem}\label{TheoremF1} Suppose that $(R,m_R)$ is a $d$-dimensional local ring and $\{I_n\}$ is a graded family of $m_R$-primary ideals in $R$. Then
\begin{enumerate}
\item[1)] The limit
$$
\lim_{n\rightarrow \infty}\frac{e(I_n)}{n^d}
$$
exists.
\item[2)] There exists a constant $\gamma$ (depending on the family $\{I_n\}$) such that
$$
e(I_{n+1})-e(I_n)\le \gamma n^{d-1}
$$
for all $n\in \NN$.
In particular, if $\{I_n\}$ is a filtration, then
$$
0\le e(I_{n+1})-e(I_n)\le \gamma n^{d-1}.
$$
\end{enumerate}
\end{Theorem}

\begin{proof} Both statements follow from the Minkowski inequality (Teissier \cite{T} and Rees and Sharp \cite{RS}),
$$
e(I_{m+n})^{\frac{1}{d}}\le e(I_mI_n)^{\frac{1}{d}}\le e(I_m)^{\frac{1}{d}}+e(I_n)^{\frac{1}{d}}
$$
for all $m,n$. In Corollary 1.5 \cite{Mus} a simple argument is given which shows that the first limit exists.
We establish the second formula. From 1), we have an upper bound 
$$
e(I_n)<cn^d,
$$
so that
\begin{equation}\label{eq60}
e(I_n)^{\frac{1}{d}}<an
\end{equation}
with 
$a=c^{\frac{1}{d}}$.
Taking $m=1$ in the Minkowski inequality, we have that
\begin{equation}\label{eq61}
e(I_{n+1})^{\frac{1}{d}}-e(I_n)^{\frac{1}{d}}\le e(I_1)^{\frac{1}{d}}.
\end{equation}
We factor
$$
\begin{array}{lll}
e(I_{n+1})-e(I_n)&=& \left(e(I_{n+1})^{\frac{1}{d}}-e(I_n)^{\frac{1}{d}}\right)
\left(e(I_{n+1})^{\frac{d-1}{d}}+e(I_{n+1})^{\frac{d-2}{d}}e(I_n)^{\frac{1}{d}}+\cdots+e(I_n)^{\frac{d-1}{d}}\right)\\
&\le& e(I_1)^{\frac{1}{d}}\left((a(n+1))^{d-1}+(a(n+1))^{d-2}an+\cdots+(an)^{d-1}\right).
\end{array}
$$
Using the inequalities (\ref{eq60}) and (\ref{eq61}), we obtain the desired bound.
\end{proof}

In contrast to Statement 2) of Theorem \ref{TheoremF1}, in any local ring $R$, there exists a graded family of $m_R$-primary ideals $\{I_n\}$ such that
$$
\limsup_{n\rightarrow\infty}\frac{e(I_{n+1})-e(I_n)}{n^{d-1}}=\infty.
$$
This is shown in equation (\ref{eq62}) of Theorem \ref{Theorem4}. Further, in any local ring $R$, there exists a filtration of
 $m_R$-primary ideals $\{J_n\}$ such that the limit
 $$
 \lim_{n\rightarrow\infty}\frac{e(J_{n+1})-e(J_n)}{n^{d-1}}
 $$
 does not exist. This is shown in 4) of Theorem \ref{Theorem7}.

\subsection{Limits of lengths of graded families of ideals}

Suppose that $\{I_n\}_{n\in\NN}$ is a graded family of $m_R$-primary ideals ($I_n$ is $m_R$-primary for $n\ge 1$) in a $d$-dimensional (Noetherian) local ring $R$. We pose the following question:

\begin{Question}  When does 
\begin{equation}\label{eq10}
\lim_{n\rightarrow\infty}\frac{\ell_R(R/I_n)}{n^d}
\end{equation}
exist?
\end{Question}

This problem was considered by Ein, Lazarsfeld and Smith \cite{ELS} and Musta\c{t}\u{a} \cite{Mus}.

Lazarsfeld and Musta\c{t}\u{a} \cite{LM} showed that
the limit exists for all graded families of $m_R$-primary ideals in $R$ if $R$ is a domain which is essentially of finite type over an algebraically closed field $k$ with $R/m_R=k$. All of these assumptions are necessary in their proof. Their proof is by reducing the problem to one on graded linear series on a projective variety, and then using a method introduced by Okounkov \cite{Ok} to reduce the problem to one of counting points in an integral semigroup.

In \cite{C4}, it is shown that
the limit exists for all graded families of $m_R$-primary ideals in $R$ if $R$ is analytically unramified ($\hat R$ is reduced), equicharacteristic and $R/m_R$ is perfect.

The nilradical $N(R)$ of a $d$-dimensional ring $R$ is 
$$
N(R)=\{x\in R\mid x^n=0 \mbox{ for some positive integer $n$}\}.
$$
Recall that 
$$
\dim N(R)=\dim R/\mbox{ann}(N(R)),
$$
so that $\dim N(R)=d$ if and only if there exists a minimal prime $P$ of $R$ such that $\dim R/P =d$ and $R_P$ is not reduced.

We now state our general theorem, which gives necessary and sufficient conditions on a local ring $R$ for all limits of graded families of $m_R$-primary ideals to exist.

\begin{Theorem}\label{Theorem1} (Theorem 5.3  \cite{C6}) Suppose that $R$ is a $d$-dimensional  local ring. Then the limit 
$$
\lim_{n\rightarrow \infty}\frac{\ell_R(R/I_n)}{n^d}
$$
exists for all 
graded families of $m_R$-primary ideals $\{I_n\}$ of $R$ if and only if
$\dim N(\hat R)<d$. 
\end{Theorem}

If $R$ is excellent, then $N(\hat R)= N(R)\hat R$, and the theorem is true with the condition $\dim N(\hat R)<d$ replaced with 
$\dim N(R)<d$. However, there exist Noetherian local domains $R$ (so that $N(R)=0$) such that $\dim N(\hat R)=\dim R$ (Nagata (E3.2) \cite{N1}).

The fact that $N(R)=d$ implies there exists a graded family without a limit was observed by Dao and Smirnov (Theorem 5.2 \cite{C6}).

The proof of Theorem \ref{Theorem1} will be given in Section \ref{Section5}.

\section{Applications}\label{Section3}

In this section, we give some applications of Theorem \ref{Theorem1}, and the method of its proof.

We first give a general ``Volume = Multiplicity'' formula.

\begin{Theorem}\label{Theorem2}(Theorem 11.5 \cite{C6}) Suppose that $R$ is a $d$-dimensional, analytically unramified local ring, and
$\{I_n\}$ is a graded family of $m_R$-primary ideals in $R$. Then 
$$
\lim_{n\rightarrow \infty}\frac{\ell_R(R/I_n)}{n^d/d!}=\lim_{p\rightarrow \infty}\frac{e_{I_p}(R)}{p^d}
$$
\end{Theorem}

Volume = Multiplicity formulas have been proven by Ein,  Lazarsfeld and Smith \cite{ELS}, Musta\c{t}\u{a} \cite{Mus} and by Lazarsfeld and Musta\c{t}\u{a} \cite{LM}. This last paper proves the formula when $R$ is essentially of finite type over an algebraically closed field $k$ with $R/m_R=k$. All of these assumptions are necessary in their proof.

The following theorem generalizes the Minkowski inequality for powers of $m_R$-primary ideals $I$ and $J$ 
found by Teissier \cite{T} and Rees and Sharp \cite{RS} (c.f. Section 17.7 \cite{SH}) to arbitrary graded families of $m_R$-primary ideals. Theorem \ref{TheoremM7} was proven for graded families of $m_R$-primary ideals in a regular local ring with algebraically closed residue field by Musta\c{t}\u{a} (Corollary 1.9 \cite{Mus}) and more recently by Kaveh and Khovanskii (Corollary 6.10 \cite{KK}).

\begin{Theorem}\label{TheoremM7} Suppose that $R$ is a Noetherian local ring of dimension $d$ with $\dim N(\hat R)<d$.
Suppose that $\{I_i\}$ and $\{J_i\}$ are graded families of $m_R$-primary ideals in $R$.  Then
$$
\left(\lim_{n\rightarrow \infty}\frac{\ell_R(R/I_n)}{n^d}\right)^{\frac{1}{d}}
+\left(\lim_{n\rightarrow \infty}\frac{\ell_R(R/J_n)}{n^d}\right)^{\frac{1}{d}}\ge
\left(\lim_{n\rightarrow \infty}\frac{\ell_R(R/I_nJ_n)}{n^d}\right)^{\frac{1}{d}}.
$$
\end{Theorem}

The proof of Theorem \ref{TheoremM7} will be given in Section \ref{Section6}. The following theorem has many consequences, as
will be discussed below.

\begin{Theorem}\label{Theorem14}(Theorem 11.1 \cite{C6}) Suppose that $R$ is an analytically unramified  local ring of dimension $d>0$. 
Suppose that $\{I_i\}$ and $\{J_i\}$ are graded families of ideals in $R$. Further suppose that $I_i\subset J_i$ for all $i$ and there exists $c\in\ZZ_+$ such that
\begin{equation}\label{eqF60}
m_R^{ci}\cap I_i= m_R^{ci}\cap J_i
\end{equation}
 for all $i$. 
 Assume that if $P$ is a minimal prime of $R$ then $I_1\subset P$ implies $I_i\subset P$ for all $i\ge 1$.
 Then the limit
$$
\lim_{i\rightarrow \infty} \frac{\ell_R(J_i/I_i)}{i^{d}}
$$
exists.
\end{Theorem}

Suppose that $R$ is a (Noetherian) local ring and $I,J$ are ideals in $R$. The generalized symbolic power $I_n(J)$ is defined by
$$
I_n(J)=I^n:J^{\infty}=\cup_{i=1}^{\infty}I^n:J^i.
$$

\begin{Theorem}(Corollary 11.4 \cite{C6}) Suppose that $R$ is an analytically unramified $d$-dimensional local ring. Let $s$ be the constant limit dimension
$s=\dim I_n(J)/I^n$ for $n\gg 0$. Suppose that $s<d$. Then 
$$
\lim_{n\rightarrow \infty}\frac{e_{m_R}(I_n(J)/I^n)}{n^{d-s}}
$$
exists.
\end{Theorem}

This theorem was proven by Herzog, Puthenpurakal and Verma \cite{HPV} for ideals $I$ and $J$ in a $d$-dimensional local ring, with the assumption that
$\bigoplus_{n\ge 0}I_n(J)$ is a finitely generated $R$-algebra.

If $R$ is a local ring and $I$ is an ideal in $R$ then the saturation of $I$ is 
$$
I^{\rm sat}=I:m_R^{\infty}=\cup_{k=1}^{\infty}I:m_R^k.
$$

\begin{Corollary}\label{Corollary5}(Corollary 11.3 \cite{C6}) Suppose that $R$ is an analytically unramified local ring of dimension $d>0$ and  $I$ is an ideal in $R$. Then the limit
$$
\lim_{i\rightarrow \infty} \frac{\ell_R((I^i)^{\rm sat}/I^i)}{i^{d}}
$$
exists.

\end{Corollary}

We now prove a very general theorem on epsilon multiplicity of modules over a local ring.
The proof requires a significant extension of Theorem \ref{Theorem14}.
The problem itself arises in the work of Kleiman, Ulrich and Validashti on equisingularity \cite{KUV}. Some preliminary parts of this work are in  \cite{UV} and \cite{Kl}. They define epsilon multiplicity as
a limsup, although it is of interest to know that it is in fact a limit. It is known by their work that the epsilon multiplicity is finite under very general conditions.

We show in Theorem \ref{TheoremC} that the epsilon multiplicity  actually exists as a limit, for modules over an 
arbitrary analytically unramified local ring $R$. 
 This includes the case when $R$ is the local ring of a reduced analytic space, which is of importance in singularity theory.

Suppose that $R$ is a $d$-dimensional,  analytically unramified local ring, and $E$ is a rank $e$
submodule of a free (finite rank) $R$-module $F=R^n$.
Let $B=R[F]$ be the symmetric algebra of $F$ over $R$, which is isomorphic to a standard graded polynomial ring $B=R[x_1,\ldots,x_n]=\bigoplus_{k\ge 0}F^k$  over $R$. We may identify $E$ with a submodule $E^1$ of $B^1$, and let $R[E]=\bigoplus_{n\ge 0}E^k$ be the $R$-subalgebra of $B$ generated by $E^1$ over $R$.

The epsilon multiplicity of $E$ is defined in \cite{UV} to be
$$
\epsilon(F/E)=\limsup_k\frac{(d+e-1)!}{k^{d+e-1}}\ell_R(H^0_{m_R}(F^k/E^k)).
$$
In the upcoming work of Kleiman, Ulrich and Validashti \cite{KUV}, it is shown that
$\epsilon(F/E)<\infty$ under very mild conditions; in particular,  the epsilon multiplicity is finite with the assumptions of Theorem \ref{TheoremC}.

We have natural $R$-module isomorphisms
$$
H^0_{m_R}(F^k/E^k)\cong E^k:_{F^k}m_R^{\infty}/E^k
$$
for all $k$.

When $I$ is an ideal in a local ring $R$, 
$$
(I^n)^{\rm sat}/I^n\cong H^0_{m_R}(R/I^n),
$$
and the epsilon multiplicity  is then
$$
\epsilon(I)=\limsup \frac{\ell_R(H^0_{m_R}(R/I^n))}{n^d/d!}.
$$
Thus Corollary \ref{Corollary5} shows that the epsilon multiplicity exists as a limit for an ideal $I$ in an arbitrary analytically unramified local ring.

An example in \cite{CHST} shows that even in the case when $E$ is an ideal $I$ in a regular local ring $R$, the  limit may be irrational.

\begin{Theorem}\label{TheoremC} Suppose that $R$ is a $d$-dimensional  analytically unramified local ring, and $E$ is a rank $e$
submodule of a free (finite rank) $R$-module $F=R^n$.
Let $B=R[F]$ be the symmetric algebra of $F$ over $R$, which is isomorphic to a standard graded polynomial ring $B=R[x_1,\ldots,x_n]=\bigoplus_{k\ge 0}F^k$  over $R$. We may identify $E$ with a submodule $E^1$ of $B^1$, and let $R[E]=\bigoplus_{n\ge 0}E^k$ be the $R$-subalgebra of $B$ generated by $E^1$ over $R$.  Suppose that $\epsilon(F/E)<\infty$. Then the limit
$$
\epsilon(F/E)=\lim_{k\rightarrow\infty} \frac{\ell_R(E^k:_{F^k}m_R^{\infty}/E^k)}{k^{d+e-1}}\in \RR
$$
exists as a limit.
\end{Theorem}

Let $P_1,\ldots, P_s$ be the minimal primes of $R$. Since $R$ is reduced,   $\cap P_i=(0)$ and the total quotient field of $R$ is isomorphic to 
$\bigoplus_{i=1}^sR_{P_i}$, with $R_{P_i}\cong (R/P_i)_{P_i}$. The assumption that $\mbox{rank}(E)=e$ is simply that $E\otimes_RR_{P_i}$ has rank $e$ for all $i$.

Theorem \ref{TheoremC} has been proven with the additional assumption that $R$ is essentially of finite type over a field of characteristic zero in many cases, in the following papers. It was first proven 
 in the case when $E=I$ is a homogeneous ideal and $F=R$ is a  standard graded normal $k$-algebra in our paper \cite{CHST} with H\`a, Srinivasan and Theodorescu. 
It is proven when $R$ is regular, essentially of finite type over a field of characteristic zero, $E=I$ is an ideal in $F=R$, and the singular locus of $\mbox{Spec}(R/I)$ is the maximal ideal in our paper \cite{CHS} with Herzog and Srinivasan. 
 
Kleiman \cite{Kl}   proved Theorem \ref{TheoremC} with the assumptions that 
$R$ is normal, essentially of finite type over an algebraically closed  field $k$ of characteristic zero with $R/m_R=k$  and  with the additional assumption that   $E$ is a direct summand of $F$ locally at every nonmaximal prime of $R$. 
Kleiman makes ingenious use of Grassmanians in his proof.

Theorem \ref{TheoremC} is proven with the additional assumptions that 
$R$ is essentially of finite type over a field $k$ of characteristic zero and $R$ has depth $\ge 2$ in \cite{C2}.

  The theorem is proven for $E$ of low analytic deviation in \cite{CHS}, for the case of ideals, and by Ulrich and Validashti \cite{UV}  for the case of modules; in the case of low analytic deviation, the limit is always zero. A generalization of this problem to the case of saturations with respect to non $\mathfrak m$-primary ideals is investigated by Herzog, Puthenpurakal and Verma in \cite{HPV}; they show that an appropriate limit exists for monomial ideals.

\vskip .2truein

We now turn to consideration of limits for families of ideals defined by valuations.
Suppose that $R$ is a Noetherian local domain with quotient field $K$.  A valuation $\nu$ of $K$ is divisorial if the valuation ring $V_{\nu}$ of $\nu$ is essentially of finite type over $R$.  A valuation $\nu$ of $K$ dominates $R$ if $R\subset V_{\nu}$ ($\nu$ is nonnegative on $R$) and $m_{\nu}\cap R=m_R$.

\begin{Lemma}\label{LemmaRees} Suppose that $R$ is an analytically irreducible local ring and $\omega$ is a divisorial valuation of the quotient field of $R$ which dominates $R$. Let
$$
I_n(\omega)=\{f\in R\mid \nu(f)\ge n\}.
$$
Then there exists a positive integer $c$ such that 
$$
m_R^{cn}\subset I_n(\omega)
$$
for all $n\in \NN$.
\end{Lemma}

\begin{proof}
 Let $\{\nu_i\}$ be the Rees valuations of $m_R$. The $\nu_i$ extend uniquely to the Rees valuations of $m_{\hat R}$. By Rees' version of Izumi's theorem, \cite{Re}, the topologies defined on $R$ by $\omega$ and the $\nu_i$ are linearly equivalent. Let $\overline \nu_{m_R}$ be the reduced order of $m_R$. By the Rees valuation theorem (recalled in \cite{Re}),
$$
\overline \nu_{m_R}(x)=\min_i\left\{\frac{\nu_i(x)}{\nu_i(m_R)}\right\}
$$
for all $x\in R$, so the topology defined by $\omega$ on $R$ is linearly equivalent to the topology defined  by $\overline\nu_{m_R}$.
The $\overline \nu_{m_R}$ topology is linearly equivalent to the $m_R$-topology by \cite{R1}, since $R$ is analytically unramified.
Thus the lemma  is established.
\end{proof}
 
\begin{Theorem}\label{Theorem100} Suppose that $R$ is an  analytically irreducible local ring of dimension $d$ and $\nu_1,\ldots,\nu_r$ are divisorial valuations of the quotient field of $R$, such that each $\nu_i$ is nonnegative on $R$. Suppose that $a_1,\ldots,a_r\in \NN$ and let
$$
I_n=\{f\in R\mid \nu_i(f)\ge a_in\mbox{ for }1\le i\le r\}.
$$
Then the limit
$$
\lim_{n\rightarrow\infty}\frac{\ell_R(I_n^{\rm sat}/I^n)}{n^d}
$$
exists.
\end{Theorem}

\begin{proof} For $1\le i\le r$, let
$$
p_i=I_1(\nu_i).
$$
By the definition of a valuation, the $p_i$ are prime ideals, and the ideals $I_n(\nu_i)$ are $p_i$-primary for all $n$ and $i$. We can reindex  the $\nu_i$ if necessary so that $p_i= m_R$ for $s< i\le r$ and $p_i\ne m_R$ for $1\le i\le s$.
By Lemma \ref{LemmaRees}, there exists a positive integer $c$ such that 
$$
m_R^{cn}\subset I_{na_i}(\nu_i)
$$
for $1\le i\le s$. Thus for all $n\in \NN$,
$$
I_n\cap m_R^{cn}=\left(\cap_{i=1}^r I_{a_in}(\nu_i)\right)\cap m_R^{cn}=\left(\cap_{i=1}^sI_{a_in}(\nu_i)\right)\cap m_R^{cn} =I_n^{\rm sat}\cap m_R^{cn}.
$$
We now apply Theorem \ref{Theorem14} to obtain the conclusions of this Theorem.
\end{proof}

The following corollary generalizes to excellent normal local rings the proposition on page 2 of \cite{Ful}, which shows that the ``local volume'' of a line bundle exists as a limit when $R$ is the local ring of a closed point of a normal  algebraic variety over an algebraically closed field.

\begin{Corollary}\label{CorollaryM40} Suppose that $R$ is an excellent normal local ring of dimension $\ge 2$, $\pi:X\rightarrow \mbox{spec}(R)$ is a proper birational map
with $X$ being normal, and $D$ is a Weil divisor on $X$. Then 
$$
\lim_{n\rightarrow\infty} \frac{\ell_R(H^1_{m_R}(\Gamma(X,\mathcal O_X(nD))))}{n^d}
$$
exists.
\end{Corollary}

\begin{proof} Let $K$ be the quotient field of $R$. We make use of the fact that if $F$ is a Weil divisor on $X$, then there is an associated rank 1 reflexive sheaf on $X$ which is denoted by $\mathcal O_X(F)$. It has the property that
$$
\Gamma(X,\mathcal O_X(F))=\{g\in K\mid (g)_X+F\ge 0\},
$$
where $(g)_X$ is the divisor of $G$ on $X$. Let $f\in \Gamma(X,\mathcal O_X(-D))$ be nonzero. Then
$f\mathcal O_X(D)=\mathcal O_X(-E)$ for some effective divisor $E$ on $X$, and thus there are induced $R$-module isomorphisms
$$
\Gamma(X,\mathcal O_X(nD))\stackrel{f^n}{\rightarrow} \Gamma(X,\mathcal O_X(-nE))
$$
for all $n$. Thus the local cohomology modules $H^1_{m_R}(\Gamma(X,\mathcal O_X(nD))$ and 
$H^1_{m_R}(\Gamma(X,\mathcal O_X(-nE))$ are isomorphic as $R$-modules for all $n$.
We have that $I_n:=\Gamma(X,\mathcal O_X(-nE))\subset R$ for all $n\in \NN$, since $R$ is normal and $E$ is effective.
Thus $\{I_n\}$ is a graded family of ideals on $R$. Since $R$ has depth $\ge 2$, we have that $H^1_{m_R}(I_n)\cong I_n^{\rm sat}/I_n$ for all $n$. Writing $E=\sum_{i=1}^r a_iE_i$ where $a_i\in\NN$ and $E_i$ are prime divisors on $X$ for $1\le i\le r$,
we let $\nu_i$ be the divisorial valuation of $K$ associated to $E_i$.
Then we see that
$$
I_n=\{g\in R\mid \nu_i(g)\ge a_i\mbox{ for }1\le i\le r\}.
$$
The existence of the limit now follows from Theorem \ref{Theorem100}, since an excellent, normal local ring is analytically irreducible (Scholie 7.8.3 (v) \cite{EGAIV}).
\end{proof}

\section{Limits of lengths of filtrations of ideals and first differences}\label{Section4}

In this section, we explore the possibility of further polynomial like behavior of the length $\ell_R(R/I_n)$ for large $n$.
Our basic conclusion is that Theorem \ref{Theorem1} is the strongest general statement in this direction that is true. 

For $x\in\RR$, $\lceil x\rceil$ is the smallest integer $n$ such that $x\le n$.

\begin{Theorem}\label{TheoremN1} Suppose that $R$ is a local ring of dimension $d>0$ with nilradical $N(R)$. Suppose that for any filtration 
$\{I_n\}$ of $R$ by $m_R$-primary ideals in $R$,
there exists some arithmetic sequence $\{am+b\}$ such that
the limit
 $$
\lim_{m\rightarrow \infty}\frac{\ell_R(R/I_{am+b})}{(am+b)^d}
$$
exists. 
 Then $\dim N(R)<d$.
\end{Theorem}

\begin{proof} Given any sequence of integers 
$$
1=s_1<s_2<\cdots<s_l<\cdots
$$
define $b_1=1$ and for any  positive integer $m$ with $s_i<m\le s_{i+1}$ define
$$
b_m=\left\{\begin{array}{ll}
b_{s_i}&\mbox{ if $i$ is odd}\\
b_{s_i}+(m-s_i)&\mbox{ if $i$ is even}
\end{array}
\right.
$$
We have that
\begin{equation}\label{eq30}
b_{m+1}\ge b_m
\end{equation}
for all $m$ and
\begin{equation}\label{eq31}
m+b_n\ge b_{m+n}
\end{equation}
for all $m,n$. Now we inductively choose the $s_i$ sufficiently large so that
\begin{equation}\label{eq32}
\frac{b_{s_{i+1}}}{s_{i+1}}=\frac{b_{s_i}}{s_{i+1}}<\frac{1}{i}\mbox{ if $i$ odd}
\end{equation}
and
\begin{equation}\label{eq33}
\frac{b_{s_{i+1}}}{s_{i+1}}=\frac{b_{s_i}+(s_{i+1}-s_i)}{s_{i+1}}>\frac{1}{2}\mbox{ if $i$ even}.
\end{equation}
Then
$$
\liminf_{i\rightarrow \infty}\frac{b_i}{i}=0\mbox{ but }\limsup_{i\rightarrow\infty}\frac{b_i}{i}\ge \frac{1}{2}.
$$
In particular,
\begin{equation}\label{eq34}
\lim_{n\rightarrow\infty}\frac{b_n}{n}
\end{equation}
does not exist, even when $n$ is constrained to lie in an arbitrary arithmetic sequence.

Let $N=N(R)$. Suppose that $\dim N=d$. Let $p$ be a minimal  prime of $N$ such that $\dim R/p=d$. Then $N_p\ne 0$, so $p_p\ne 0$ in $R_p$. $p$ is an associated prime of $N$,  so there exists $0\ne x\in R$ such that $\mbox{ann}(x)=p$. $x\in p$, since otherwise 
$0=pxR_p=p_p$ which is impossible. In particular, $x^2=0$.

Define $m_R$-primary ideals in $R$ by
$$
I_n=m_R^n+xm_R^{b_n}.
$$
$\{I_n\}$ is a graded family of $m_R$-primary ideals by (\ref{eq31})in $R$ since 
$$
I_mI_n=(m_R^{m+n},xm_R^{m+b_n}, xm_R^{n+b_m}),
$$
and is a filtration by (\ref{eq30}).
Let $\overline R=R/xR$. We have short exact sequences
\begin{equation}\label{eqnr2}
0\rightarrow xR/xR\cap I_n\rightarrow R/I_n\rightarrow \overline R/I_n\overline R\rightarrow 0.
\end{equation}
By Artin-Rees, there exists a number $k$ such that
$xR\cap m_R^n=m_R^{n-k}(xR\cap m_R^{n-k})$ for $n>k$. Thus $xR\cap m_R^n\subset xm_R^{b_n}$ for $n\gg 0$ and $xR\cap I_n=xm_R^{b_n}$ for $n\gg 0$.
We have that
$$
xR/xR\cap I_n\cong xR/xm_R^{b_n}\cong R/({\rm ann}(x)+m_R^{b_n})\cong R/p+m_R^{b_n},
$$
so that  $\ell_R(xR/xR\cap I_n)=P_{R/p}(b_n)$  for $n\gg0$, where $P_{R/p}(n)$ is the Hilbert-Samuel polynomial of $R/p$. Hence
\begin{equation}\label{eqnr3}
\lim_{n\rightarrow \infty}\frac{\ell_R(xR/xR\cap I_n)}{n^d}=\frac{e(m_{R/p})}{d!}\lim_{n\rightarrow\infty}\left(\frac{b_n}{n}\right)^d
\end{equation}
does not exist by (\ref{eq34}).
For $n\gg0$, 
$$
\ell_R(\overline R/I_n\overline R)=\ell_R(\overline R/m_{\overline R}^n)=P_{\overline R}(n)
$$
where $P_{\overline R}(n)$ is the Hilbert-Samuel polynomial of $\overline R$. Since $\dim \overline R\le d$, we have that
\begin{equation}\label{eqnr4}
\lim_{n\rightarrow \infty}\frac{\ell_R(\overline R/I_n\overline R)}{n^d}
\end{equation}
exists. Thus
$$
\lim_{n\rightarrow \infty}\frac{\ell_R(R/I_n)}{n^d}
$$
does not exist by (\ref{eqnr2}), (\ref{eqnr3}) and (\ref{eqnr4}).
\end{proof}

We obtain  the following variant of Theorem \ref{Theorem1}, using the fact that a filtration is a graded family of ideals.

\begin{Theorem}\label{Theorem10} Suppose that $R$ is a local ring of dimension $d>0$ and $N(\hat R)$ is the nilradical of the $m_R$-adic completion $\hat R$ of $R$. Then the limit
$$
\lim_{n\rightarrow \infty} \frac{\ell_R(R/I_n)}{n^d}
$$
exists for any filtration of $R$ by $m_R$-primary ideals $\{I_n\}$ if and only if $\dim N(\hat R)<d$
\end{Theorem}

There is a small distinction from the statement of Theorem \ref{Theorem1}. In the case when $d=\dim R=0$, 
and $\{I_n\}$ is a filtration of $R$ by $m_R$-primary ideals, there exists $n_0$ such that $I_n=I_{n+1}$ for all $n\ge n_0$,
since $I_{n+1}\subset I_n$ for all $n$ and $R$ has finite length. Thus
$$
\lim_{n\rightarrow \infty}\ell_R(R/I_n)
$$
always exists if $\{I_n\}$ is a filtration by $m_R$-primary ideals of a zero dimensional local ring. 
However, such limits do not always exist on a 0-dimensional local ring $R$  if $R$ is not reduced (Theorem 5.5 \ref{Theorem1}).

\begin{Theorem}\label{Theorem4} Suppose that $(R,m_R)$ is a  local ring of dimension $d>0$. Then there exists a graded family of $m_R$-primary ideals
$\{I_n\}$ in $R$ such that
\begin{equation}\label{eqF61}
\limsup_{n\rightarrow \infty}\frac{\ell_R(R/I_{m})-\ell_R(R/I_{m+1})}{m^{d-1}}=\infty
\end{equation}
and
\begin{equation}\label{eq62}
\limsup_{n\rightarrow\infty}\frac{e(I_m)-e(I_{m+1})}{m^{d-1}}=\infty.
\end{equation}
\end{Theorem}

\begin{proof} Inductively define a function $\sigma:\ZZ_+\rightarrow \ZZ_+$ by
$\sigma(1)=2$, 
$$
\begin{array}{l}
\mbox{$\sigma(m)=\sigma(m-1)$ if $m\ne 2^{2^n}$ for some $n\in \ZZ_+$ and}\\
\mbox{$\sigma(m)=\sigma(m-1)-\frac{1}{2^n}$ if $m=2^{2^n}$ for some $n\in \ZZ_+$.}
\end{array}
$$
We have that
$$
\sum_{n=1}^{\infty}\frac{1}{2^n}=1,
$$
so 
\begin{equation}\label{eq41}
1\le \sigma(i)\le 2
\end{equation}
 for all $i$.  We have that $\sigma(a)\ge \sigma(b)$ if $b>a$, so
$$
\begin{array}{lll}
\lceil m_1\sigma(m_1)\rceil+\lceil m_2\sigma(m_2)\rceil&\ge &m_1\sigma(m_1)+m_2\sigma(m_2)\\
&\ge& (m_1+m_2)\min\{\sigma(m_1),\sigma(m_2)\}\\
&\ge&(m_1+m_2)\sigma(m_1+m_2)
\end{array}
$$
so the integers
\begin{equation}\label{eq40}
\lceil m_1\sigma(m_1)\rceil+\lceil m_2\sigma(m_2)\rceil\ge \lceil (m_1+m_2)\sigma(m_1+m_2)\rceil.
\end{equation}
Let $I_0=R$ and
$$
I_m=m_R^{\lceil m\sigma(m)\rceil}
$$
for $m\ge 1$. $\{I_m\}$ is a graded family of $m_R$-primary ideals in $R$ by (\ref{eq40}).

Let 
$$
P_R(t)=\frac{e(R)}{d!}t^d+\mbox{ lower order terms in $t$}
$$ 
be the Hilbert polynomial of $R$. For $m\gg 0$, we have that
$$
\ell_R(R/I_m)=P_R(\lceil m\sigma(m)\rceil).
$$
Let
$$
F(m)=\frac{\ell_R(R/I_m)-\ell_R(R/I_{m+1})}{m^{d-1}}.
$$
Using the bound (\ref{eq41}), the bound
$$
m\sigma(m)\le \lceil m\sigma(m)\rceil\le m\sigma(m)+1
$$
for all $m$, and since $P_R(t)$ is a polynomial of degree $d$, there exists a positive constant $c$ such that
$$
F(m)\ge \frac{e(R)}{d!}\left(\frac{(m\sigma(m))^d-((m+1)\sigma(m+1)+a)^d}{m^{d-1}}\right)-c
$$
for $m\gg 0$. Expanding
$$
\begin{array}{lll}
((m+1)\sigma(m+1)+1)^d&=& (m\sigma(m+1)+(\sigma(m+1)+1))^d\\
&=&\sum_{i=1}^d\binom{d}{i}(m\sigma(m+1))^i(\sigma(m+1)+1)^{d-i},
\end{array}
$$
we see that there exists a positive integer $c'$ such that 
$$
F(m)\ge \frac{e(R)}{d!}m(\sigma(m)^d-\sigma(m+1)^d)-c'
$$
for $m\gg 0$. Let $m+1=2^{2^n}$. We have that
$$
\sigma(m+1)^d=(\sigma(m)-\frac{1}{2^n})^d
=\sum_{i=0}^d(-1)^i\left(\frac{1}{2^n}\right)^i\sigma(m)^{d-i},
$$
so that
$$
m(\sigma(m)^d-\sigma(m+1)^d)
=\frac{m}{2^n}\left(d!\sigma(m)^{d-1}-\sum_{i=2}^d(-1)^i\left(\frac{1}{2^n}\right)^i\sigma(m)^{d-i}\right).
$$
Thus there exists a positive constant $\lambda$ such that 
$$
F(m)\ge \lambda \frac{m}{2^n}-c'
$$
for $m=2^{2^n}-1\gg 0$. Writing
$$
\frac{m}{2^n}=\frac{2^{2^n}-1}{2^n},
$$
we see that $F(m)$ goes to infinity for large $m$ like $\frac{m}{{\rm log}_2m}$, giving us the formula (\ref{eqF61}).
We have that
$$
e(I_m)=e(R)(\lceil m\sigma(m)\rceil)^d
$$
for all $m\in \ZZ_+$, so the above calculation also gives us the formula (\ref{eq62}).

\end{proof}

\begin{Lemma}\label{Lemma1}
Suppose that $S=k[x_1,\ldots,x_n]$ is a polynomial ring over a field $k$ and $I\subset S$ is an ideal which is generated by monomials. Let $N=(x_1,\ldots,x_n)$, an ideal in $S$. Suppose that $r,s\in \ZZ_+$ and $N^s\subset I$. Then
$$
\dim_k(I/N^rI)\le (s+r)^{d-1}r.
$$
\end{Lemma} 

\begin{proof} Let 
$$
(\RR^d)_+=\{(a_1,\ldots,a_d)\in \RR^d\mid a_i\ge 0\mbox{ for all }i\}.
$$
Given an ideal $J$ in $S$ which is generated by monomials, let
$$
\mbox{NP}(J)=\cup \left((a_1,\ldots,a_d)+(\RR^d)_+\right)
$$
where the union in $\RR^d$ is over all $(a_1,\ldots,a_d)$ such that $x_1^{a_1}x_2^{a_2}\cdots x_d^{a_d}\in J$.
Let $\pi:\RR^d\rightarrow \RR^{d-1}$ be projection onto the first $d-1$ factors. Let $T$ be the simplex
$$
T=\{(y_1,\ldots,y_{d-1})\in \RR^{d-1}\mid y_1,\ldots,y_{d-1}\ge 0\mbox{ and }y_1+\cdots+y_{d-1}\le s+r\}
$$
in $\RR^{d-1}$. Let
$$
U=\mbox{NP}(I)\setminus \mbox{NP}(N^rI)\subset \RR^d.
$$
We have that
$$
\dim_k(I/N^rI)=\#(U\cap \ZZ^d).
$$
$N^s\subset I$ implies $N^{s+r}\subset N^rI$, so
$$
\mbox{NP}(N^{s+r})\subset\mbox{NP}(N^rI)\subset \mbox{NP}(I).
$$
Thus for $w\in \RR^{d-1}$, $U\cap \pi^{-1}(w)=\emptyset$ if $w\not\in T$. 
$$
(0,0,\ldots,0,r)+\mbox{NP}(I)\subset\mbox{NP}(N^rI)
$$
so $\#(\pi^{-1}(w)\cap U\cap \ZZ^d)\le r$ if $w\in T\cap \ZZ^{d-1}$. Thus
$$
\#(U\cap \ZZ^d)\le \#(T)r\le (s+r)^{d-1}r.
$$
\end{proof}

\begin{Theorem}\label{Theorem6} Suppose that $(R,m_R)$ is a regular local ring of dimension $d>0$ and $\{I_n\}$ is a filtration of $R$ by
$m_R$-primary ideals. Then there exists a constant $\gamma>0$ such that 
$$
0\le \ell_R(I_n/I_{n+1})<\gamma n^{d-1}
$$
for all $n$.
\end{Theorem}

\begin{proof} There exists a positive integer $c$ such that $m_R^c\subset I_1$, so that $m_R^{cn}\subset I_n$ for all $n$.
We further have that $I_1I_n\subset I_{n+1}$, so $m_R^cI_n\subset I_{n+1}$, and thus
$$
\ell_R(I_n/I_{n+1})\le \ell_R(I_n/m_R^cI_n).
$$
Let $k=R/m_R$. Since $R$ is regular,
$$
S={\rm gr}_{m_R}(R)=\bigoplus_{n\ge 0}m_R^n/m_R^{n+1}=k[x_1,\ldots,x_d]
$$
is a standard graded polynomial ring, where $x_1,\ldots,x_d$ are the initial forms of a regular system of parameters in $R$.
Let $M=(x_1,\ldots,x_d)$. Let $J={\rm in}(I_n)$ be the initial ideal of $I_n$. The initial ideal of $m_R^l$ is ${\rm in}(m_R^l)=M^n$ for all $n$. We have that
$$
M^{cn}\subset J
$$ 
and
$$
M^cJ={\rm in}(m_R^c){\rm in}(I_n)\subset{\rm in}(m_R^cI_n)\subset J.
$$
Thus
$$
\ell_R(I_n/I_{n+1})\le \dim_k(S/M^cJ)-\dim_k(S/J).
$$
On the polynomial ring $S=k[x_1,\ldots,x_d]$, we can refine the grading by the degree lex order deglex. Let
$$
A={\rm gr}_{deglex}(S)\cong k[y_1,\ldots,y_d]
$$
be the associated multi-graded ring. Let $N$ be the graded maximal ideal $N=(y_1,\ldots,y_d)$.
Let $K={\rm in}_{deglex}(J)$ be the associated initial ideal in $A$. ${\rm in}_{deglex}(M^l)=N^l$ for all $l$. We have that
$$
N^cK={\rm in}_{deglex}(M^c){\rm in}_{deglex}(J)\subset {\rm in}_{deglex}(M^cJ)\subset K.
$$
We have that
$$
\dim_k(S/M^cJ)-\dim_k(S/J)\le \dim_k(A/N^cK)-\dim_k(A/K).
$$
Now take $r=c$ and $s=cn$ in Lemma \ref{Lemma1} to get
$$
\ell_R(I_n/I_{n+1})\le c^d(n+1)^{d-1}.
$$
\end{proof}

\begin{Theorem}\label{Theorem7} Suppose that $(R,m_R)$ is a  local ring of dimension $d>0$. Then there exists a filtration $\{I_n\}$ of $m_R$-primary ideals in $R$ such that the limit
\begin{enumerate}
\item[1)]
$$
v=\lim_{n\rightarrow\infty}\frac{\ell_R(R/I_n)}{n^d}
$$
exists, and is equal to $\frac{e(R)}{d!}$.
\item[2)] The function $\frac{\ell_R(R/I_n)-\lceil vn^d\rceil}{n^{d-1}}$ goes to infinity of order $\log_2(n)$.
\item[3)] The function $\frac{\ell_R(R/I_{n+1})-\ell_R(R/I_n)}{n^{d-1}}$ is  bounded, but the limit
$$
\lim_{n\rightarrow\infty}\frac{\ell_R(R/I_{n+1})-\ell_R(R/I_n)}{n^{d-1}}
$$
does not exist, even when $n$ is constrained to lie in any  arithmetic sequence $\{am+b\}_{m\in\NN}$.
\item[4)] The limit
$$
\lim_{n\rightarrow \infty}\frac{e(I_{n+1})-e(I_n)}{n^{d-1}}
$$
does not exist, even when $n$ is constrained to lie in any  arithmetic sequence $\{am+b\}_{m\in\NN}$.
\end{enumerate} 
\end{Theorem}

\begin{proof} Define a sequence $\{a_i\}_{i\in\NN}$ of positive integers by
$a_0=0$, $a_1=1$ and 
$a_i=\sigma$ if $i=2^{\sigma}+r$ with $0\le r<2^{\sigma}$ for $i\ge 2$.
We then have that
\begin{equation}\label{eq50}
a_{m+1}\ge a_m
\end{equation}
for all $m\in\NN$ and
\begin{equation}\label{eq51}
a_m+a_n\ge a_{m+n}
\end{equation}
for all $m,n\in\NN$. We give a verification of (\ref{eq51}). First assume that $n\ge m\ge 2$. Then
$m\le 2^{a_m+1}-1$ and $n\le 2^{a_n+1}-1$ so $m+n\le 2^{a_n+2}-2$. Thus
$$
a_{m+n}\le a_n+1\le a_n+a_m.
$$
Now suppose that $n\ge 2$ and $m=1$. Then $m+n\le 1+(2^{a_n+1}-1)$ so $a_{m+n}\le a_n+1\le a_n+a_m$.
The final case is when $m-n=1$. Then $m+n=2$ so $a_{m+n}=1\le a_{m}+a_n$.
\vskip .2truein
Set $b_n=n+a_n$ for $n\in \NN$. From (\ref{eq50}) and (\ref{eq51}) we conclude that
\begin{equation}\label{eq52}
b_{m+1}\ge b_m
\end{equation}
and
\begin{equation}\label{eq53}
b_m+b_n\ge b_{m+n}.
\end{equation}
For $n\ge 2$, we have that
$a_n\le \log_2n\le a_n+1$, so
\begin{equation}\label{eq54}
(\log_2n-1)\le a_n\le \log_2n.
\end{equation}
For $n\in\NN$, define $I_n=m_R^{b_n}$. $\{I_n\}_{n\in\NN}$ is a filtration of $m_R$-primary ideals in $R$
by (\ref{eq52}) and (\ref{eq53}). Let 
\begin{equation}\label{eq73}
P_R(t)=\frac{e(R)}{d!}t^d+\gamma t^{d-1}+\mbox{ lower order terms in $t$}
\end{equation}
 be the Hilbert polynomial of $R$, so that
$$
\ell_R(R/m_R^n)=P_R(n)
$$
for $n\gg 0$.

For $e,f\in\ZZ_+$, we have that
$$
\frac{b_n^e}{n^f}=\left(\frac{n+a_n}{n^{\frac{f}{e}}}\right)^e
=\left(n^{\frac{e-f}{e}}+a_nn^{-\frac{f}{e}}\right)^e.
$$
Thus 
\begin{equation}\label{eq71}
\lim_{n\rightarrow\infty} \frac{b_n^e}{n^f}=0
\end{equation}
if $f>e$ and 
\begin{equation}\label{eq72}
\lim_{n\rightarrow\infty} \frac{b_n^e}{n^e}=1.
\end{equation}
Thus the volume
$$
v=\lim_{n\rightarrow \infty}\frac{\ell_R(R/I_n)}{n^d}=\frac{e(R)}{d!}
$$
exists as a limit.

By (\ref{eq71}) and (\ref{eq72}), and using the notation of (\ref{eq73}),
\begin{equation}\label{eq55}
\begin{array}{lll}
\frac{\ell_R(R/I_n)-\lceil n^d\frac{e(R)}{d!}\rceil}{n^{d-1}}
&=&\frac{e(R)}{d!}\left(\frac{(n+a_n)^d-n^d}{n^{d-1}}\right)+\gamma\frac{b_n^{d-1}}{n^{d-1}}\\
&&+
\mbox{ terms whose absolute values become arbitrarily}\\
&&\mbox{ small for $n\gg 0$.} 
\end{array}
\end{equation}

We expand
$$
\begin{array}{lll}
\frac{(n+a_n)^d-n^d}{n^{d-1}}&=& \frac{\sum_{i=0}^{d-1}\binom{d}{i}n^ia_n^{d-i}}{n^{d-1}}\\
&=&\sum_{i=0}^{d-1}\binom{d}{i}n^{i+1-d}a_n^{d-i}.
\end{array}
$$
We see from (\ref{eq54}) that $n^{i+1-d}a_n^{d-i}\rightarrow 0$ as $n\rightarrow\infty$ if $i<d-1$, so the only significant term
in (\ref{eq55}) is $e(R)a_n$ which satisfies the bound (\ref{eq54}), so that we have verified 2). 

We now verify 3).
$$
\begin{array}{lll}
\frac{\ell_R(R/I_{n+1})-\ell_R(R/I_n)}{n^{d-1}}&=&
\frac{e(R)}{d!}\left(\frac{(n+1+a_{n+1})^d-(n+a_n)^d}{n^{d-1}}\right)
+\gamma\frac{b_{n+1}^{d-1}-b_n^{d-1}}{n^{d-1}}\\
&&+\mbox{ terms which go to zero with large $n$.}
\end{array}
$$
Now
$$
\frac{b_{n+1}^{d-1}-b_n^{d-1}}{n^{d-1}}=\sum_{i=1}^{d-1}\frac{(1+a_{n+1})^i-a_n^i}{n^i},
$$
so
$$
\begin{array}{lll}
\frac{\ell_R(R/I_{n+1})-\ell_R(R/I_n)}{n^{d-1}}&=&
\frac{e(R)}{d!}\left(\frac{\sum_{i=0}^d\binom{d}{i}n^i(1+a_{n+1})^{d-i}-\sum_{i=0}^d\binom{d}{i}n^ia_n^{d-i}}{n^{d-1}}\right)\\
&&+\mbox{ terms which go to zero with large $n$}\\
&=&\frac{e(R)}{(d-1)!}(1+a_{n+1}-a_n)+\mbox{ terms which go to zero with large $n$.}
\end{array}
$$
We have that
$$
a_{n+1}-a_n=\left\{\begin{array}{ll}
1&\mbox{ if $n=2^{\sigma}-1$ for some integer $\sigma$}\\ 0&\mbox{ otherwise}
\end{array}\right.
$$
Thus 3) follows.

Since
$$
e(I_n)=\frac{e(R)}{d!}b_n^d,
$$
the same calculation verifies 4).

\end{proof}

\section{The proof of Theorem \ref{Theorem1}}\label{Section5}

 In this section we give an outline of the proof of Theorem \ref{Theorem1}. We refer to the papers \cite{C4}, \cite{C5} and \cite{C6} for  details.

\subsection{Proof that $\bf \dim N(\hat R)<d$ implies limits exist}

Suppose that $R$ is a $d$-dimensional  local ring with $\dim N(R)<d$ and $\{I_n\}$ is a graded family of $m_R$-primary ideals in $R$.

We have that $\ell_{\hat R}(R/I_n\hat R)=\ell_R(R/I_n)$ for all $n$ so we may assume that $R=\hat R$ is complete; in particular, we may assume that $R$ is excellent with $\dim N(R)<d$. There exists a positive integer $c$ such that $m_R^c\subset I_1$, which implies that 
\begin{equation}\label{eq13}
m_R^{nc}\subset I_n
\mbox{ for all positive $n$.}
\end{equation}

Let $N=N(R)$ and $A=R/N$. We have short exact sequences
$$
0\rightarrow N/N\cap I_iR\rightarrow R/I_iR\rightarrow A/I_iA\rightarrow 0,
$$
from which we deduce that there exists a constant $\alpha>0$ such that
$$
\ell_R(N/N\cap I_iR)\le \ell_R(N/m_R^{ci}N)\le \alpha i^{\dim N}\le \alpha i^{d-1}.
$$
Replacing $R$ with $A$ and $I_n$ with $I_nA$, we thus reduce to the case that $R$ is reduced. Using the following lemma, we then reduce to the case that $R$ is a complete domain
(so that it is analytically irreducible).

\begin{Lemma}\label{LemmaM30}(Lemma 5.1  \cite{C4}) Suppose that $R$ is a $d$-dimensional reduced local domain, and $\{I_n\}$ is a graded family of $m_R$-primary ideals in $R$. Let $\{P_1,\ldots,P_s\}$ be the set of minimal primes of $R$ and let  $R_i=R/P_i$. Then there exists $\alpha>0$ such that 
$$
|(\sum_{i=1}^s\ell_{R_i}(R_i/I_nR_i))-\ell_R(R/I_n)|\le\alpha n^{d-1}
$$
for all $n$.
\end{Lemma}

We now present a method introduced by Okounkov \cite{Ok} to compute limits  of multiplicities.
The method has been refined by Lazarsfeld and Musta\c{t}\u{a} \cite{LM} and Kaveh and Khovanskii \cite{KK}.

Suppose that $\Gamma\subset \NN^{d+1}$ is a semigroup. Let $\Sigma(\Gamma)$ be the closed convex cone generated by $\Gamma$ in $\RR^{d+1}$. Define $\Delta(\Gamma)=\Sigma(\Gamma)\cap(\RR^d\times\{1\})$. For $i\in\NN$, let 
$\Gamma_i=\Gamma\cap(\NN^d\times\{i\})$.

\begin{Theorem}\label{Theorem3} (Okounkov \cite{Ok}, Lazarsfeld and Musta\c{t}\u{a} \cite{LM})
Suppose that $\Gamma$ satisfies
\begin{enumerate}
\item[1)] There exist finitely many vectors $(v_i,1)\in\NN^{d+1}$  spanning a semigroup $B\subset \NN^{d+1}$ such that $\Gamma\subset B$ (boundedness).
\item[2)] The subgroup generated by $\Gamma$ is $\ZZ^{d+1}$. 
\end{enumerate}
Then 
$$
\lim_{i\rightarrow\infty}\frac{\# \Gamma_i}{i^d}={\rm vol}(\Delta(\Gamma))
$$
exists.
\end{Theorem}

We now return to the proof that $\dim N(\hat R)<d$ implies limits exist. Recall that we have reduced to the case that $R$ is a complete domain. Let $\pi:X\rightarrow \mbox{spec}(R)$ be the normalization of the blow up of $m_R$. Since $X$ is excellent, we have that 
$X$ is of finite type over $R$. $X$ is regular in codimension 1, so there exists a closed point $p\in \pi^{-1}(m_R)$ such that $S=\mathcal O_{X,p}$ is regular and dominates $R$. We have an inclusion $R\rightarrow S$ of $d$-dimensional local rings such that $m_S\cap R=m_R$ with equality of quotient fields $Q(R)=Q(S)$. Let $k=R/m_R$, $k'=S/m_S$. Since $S$ is essentially  of finite type over $R$, we have that $[k':k]<\infty$. Let $y_1,\ldots,y_d$ be regular parameters in $S$. Choose $\lambda_1,\ldots,\lambda_d\in\RR_+$ which are rationally independent with $\lambda_i\ge 1$. Prescribe a rank 1 valuation $\nu$ on $Q(R)$ by
$$
\nu(y_1^{i_1}\cdots y_d^{i_d})=i_1\lambda_1+\cdots+i_d\lambda_d
$$
and $\nu(\gamma)=0$ if $\gamma\in S$ is a unit. The value group of $\nu$ is
$$
\Gamma_{\nu}=\lambda_1\ZZ+\cdots+\lambda_d\ZZ\subset \RR.
$$
Let $V_{\nu}$ be the valuation ring of $\nu$. Then 
$$
k'=S/m_S\cong V_{\nu}/m_{\nu}.
$$
For $\lambda\in \RR_+$, define valuation ideals in $V_{\nu}$ by
$$
K_{\lambda}=\{f\in Q(R)\mid \nu(f)\ge \lambda\}
$$
and
$$
K_{\lambda}^+=\{f\in Q(R)\mid \nu(f)>\lambda\}.
$$

Now suppose that $I\subset R$ is an ideal and $\lambda\in\Gamma_{\nu}$ is nonnegative. We have an inclusion
$$
I\cap K_{\lambda}/I\cap K_{\lambda}^+\subset K_{\lambda}/K_{\lambda^+}\cong k'.
$$
Thus
$$
\dim_kI\cap K_{\lambda}/I\cap K_{\lambda}^+\le [k':k].
$$

\begin{Lemma}\label{Prop1}(Lemma 4.3 \cite{C4})  There exists $\alpha\in\ZZ_+$ such that $K_{\alpha n}\cap R\subset m_R^n$ for all $n\in\ZZ_+$
\end{Lemma}

The proof uses Huebl's linear Zariski subspace theorem \cite{Hu} or Rees' Izumi Theorem \cite{Re}. The assumption that $R$ is analytically irreducible is necessary for the lemma. Recalling the constant $c$ of (\ref{eq13}), let $\beta=\alpha c$. We then have that
\begin{equation}\label{eqF30}
K_{\beta n}\cap R\subset m_R^{nc}\subset I_n
\end{equation}
for all $n$. For $1\le t\le [k':k]$, define

$$
\Gamma^{(t)}=\left\{(n_1,\ldots,n_d,i)\in \NN^{d+1}\mid \begin{array}{l}
\dim_k I_i\cap K_{n_1\lambda_1+\cdots+n_d\lambda_d}/I_i\cap K^+_{n_1\lambda_1+\cdots+n_d\lambda_d}\ge t\\
\mbox{and $n_1+\cdots+n_d\le\beta i$}\end{array}\right\}
$$
and

$$
\hat\Gamma^{(t)}=\left\{(n_1,\ldots,n_d,i)\in \NN^{d+1}\mid \begin{array}{l}
\dim_k R\cap K_{n_1\lambda_1+\cdots+n_d\lambda_d}/R\cap K^+_{n_1\lambda_1+\cdots+n_d\lambda_d}\ge t\\
\mbox{and $n_1+\cdots+n_d\le\beta i$}\end{array}\right\}
$$

\begin{Lemma}\label{LemmaM10}(Lemma 4.4 \cite{C6}) Suppose that $t\ge 1$, $0\ne f\in I_i$, $0\ne g\in I_j$ and 
$$
\dim_kI_i\cap K_{\nu(f)}/I_i\cap K_{\nu(f)}^+\ge t.
$$
Then
$$
\dim_kI_{i+j}\cap K_{\nu(fg)}/I_{i+j}\cap K_{\nu(fg)}^+\ge t.
$$
\end{Lemma}
Since $\nu(fg)=\nu(f)+\nu(g)$, we conclude that when they are nonempty, $\Gamma^{(t)}$ and $\hat \Gamma^{(t)}$ are subsemigroups of $\NN^{d+1}$.

Given $\lambda=n_1\lambda_1+\cdots+n_d\lambda_d$ such that $n_1+\cdots +n_d\le\beta i$, we have that
$$
\dim_k K_{\lambda}\cap I_i/K_{\lambda}^+\cap I_i
=\#\{t\mid (n_1,\ldots, n_d,i)\in \Gamma^{(t)}\}.
$$

recalling (\ref{eqF30}), we have that
\begin{equation}\label{eqM11}
\begin{array}{lll}
\ell_R(R/I_i)&=&\ell_R(R/K_{\beta i}\cap R)-\ell_R(I_i/K_{\beta i}\cap I_i)\\
&=& (\sum_{0\le \lambda<\beta i}\dim_k K_{\lambda}\cap R/K_{\lambda}^+\cap R)
-(\sum_{0\le \lambda<\beta i}\dim_k K_{\lambda}\cap I_i/K_{\lambda}^+\cap I_i)\\
&=&(\sum_{t=1}^{[k':k]}\#\hat\Gamma_i^{(t)})-(\sum_{t=1}^{[k':k]}\#\Gamma_i^{(t)})
\end{array}
\end{equation}
where 
$\Gamma_i^{(t)}=\Gamma^{(t)}\cap(\NN^d\times\{i\})$ and $\hat\Gamma_i^{(t)}=\hat\Gamma^{(t)}\cap(\NN^d\times\{i\})$.
The semigroups $\Gamma^{(t)}$ and $\hat \Gamma^{(t)}$ satisfy the hypotheses of Theorem \ref{Theorem3}. Thus
\begin{equation}\label{eqM12}
\lim_{i\rightarrow\infty}\frac{\#\Gamma_i^{(t)}}{i^d}={\rm vol}(\Delta(\Gamma^{(t)})
\end{equation}
and
\begin{equation}\label{eqM13}
\lim_{i\rightarrow\infty}\frac{\#\hat\Gamma_i^{(t)}}{i^d}={\rm vol}(\Delta(\hat\Gamma^{(t)})
\end{equation}
so that
$$
\lim_{i\rightarrow\infty}\frac{\ell_R(R/I_i)}{i^d}
$$
exists.
\section{Proof of Theorem \ref{TheoremM7}}\label{Section6}
We begin by refining our calculation of the limit in the previous section. Let notation be as in the previous section.
Forgetting the 1 in the $(d+1)^{st}$ component, we may regard $\Delta(\Gamma^{(t)})$ and $\Delta(\hat\Gamma^{(t)})$ as subsets of $\RR^d$.

\begin{Lemma}\label{LemmaM1} Suppose that $\Gamma^{(t)}\ne \emptyset$. Then $\Delta(\Gamma^{(t)})=\Delta(\Gamma^{(1)})$.
\end{Lemma}

\begin{proof} We have that $\Gamma^{(t)}\subset \Gamma^{(1)}$ so $\Delta(\Gamma^{(t)})\subset \Delta(\Gamma^{(t)})$.

Suppose that  $(l_1,\ldots,l_d,i)\in \Gamma^{(1)}$. We must show that there exists $(m_1,\ldots,m_d,j)\in \Gamma^{(t)}$ such that $\|u-v\|<\epsilon$ where
$$
u=\frac{1}{i}(l_1,\ldots,l_d)\in \Delta(\Gamma^{(1)})\mbox{ and }v=\frac{1}{j}(m_1,\ldots,m_d)\in \Delta(\Gamma^{(t)}).
$$
By assumption, there exists $(n_1,\ldots,n_d,k)\in \Gamma^{(t)}$. Let
$$
w=\frac{1}{k}(n_1,\ldots,n_d)\in \Delta(\Gamma^{(t)}).
$$
$$
(sl_1+n_1,sl_2+n_2,\ldots,sl_d+n_d,si+k)\in \Gamma^{(t)}
$$
for all $s\in \NN$ by Lemma \ref{LemmaM10}. Thus
$$
v=\frac{1}{si+k}(sl_1+n_1,sl_2+n_2,\ldots,sl_d+n_d)\in \Delta(\Gamma^{(t)}).
$$
We write
$$
v=\left(\frac{1}{1+\frac{k}{si}}\right)u+\frac{k}{si+k}w
$$
which is arbitrarily close to $u$ for $s$ sufficiently large.
\end{proof}

The same argument shows that
\begin{equation}\label{eqM3}
\Delta(\hat\Gamma^{(t)})=\Delta(\hat\Gamma^{(t)})
\end{equation}
if $\hat\Gamma^{(t)}\ne\emptyset$.

\begin{Lemma}\label{LemmaM2} We have that $\hat\Gamma^{(t)}\ne\emptyset$ for $1\le t\le [k':k]$.
\end{Lemma}
\begin{proof} Let $s=[k':k]$ and let $f_1,\ldots,f_s\in Q(R)$ be such that their classes in $V_{\nu}/m_{\nu}$ are a $k$-basis of $k'=V_{\nu}/m_{\nu}$. There exist $g_1,\ldots,g_s,h\in R$ such that 
$$
f_i=\frac{g_i}{h}\mbox{ for }1\le i\le s.
$$
Let $\lambda=\nu(h)$. Then $\nu(g_i)=\lambda$ for $1\le i\le s$ since $\nu(f_i)=0$. Suppose that 
$$
\nu(c_1g_1+\cdots+c_sg_s)>\lambda
$$
for some $c_1,\ldots,c_s\in R$. Let $b=c_1g_1+\cdots+c_sg_s$. 
$$
c_1f_1+\cdots+c_sf_s=\frac{b}{h}
$$
and $\nu(\frac{b}{h})>0$, so all for all $i$, $c_i\in m_{\nu}\cap R=m_R$. Thus the classes of $g_1,\ldots,g_s$ are linearly independent over $k$ in $R\cap K_{\lambda}/R\cap K^+_{\lambda}$, and so $\hat \Gamma^{(s)}\ne \emptyset$.
\end{proof}

We deduce from Lemma \ref{LemmaM2} and Lemma \ref{LemmaM10} that 
\begin{equation}\label{eqM4}
\Gamma^{(t)}\ne \emptyset\mbox{ for }1\le t\le [k':k].
\end{equation}

We obtain the following refinement of Theorem \ref{Theorem1}.

\begin{Theorem}\label{TheoremM5} Suppose that $R$ is a $d$-dimensional analytically irreducible noetherian local ring, and $\{I_n\}$ is a graded family of $m_R$-primary ideals in $R$. Then
$$
\lim_{n\rightarrow\infty} \frac{\ell_R(R/I_n)}{n^d}=[k':k]\left({\rm vol}(\Delta(\hat\Gamma^{(1)}))-{\rm vol}(\Delta(\Gamma^{(1)}))\right).
$$
\end{Theorem} 

\begin{proof} The proof follows from equations (\ref{eqM11}), (\ref{eqM12}), (\ref{eqM13}), Lemma \ref{LemmaM1} and equation (\ref{eqM3}), Lemma \ref{LemmaM2} and (\ref{eqM4}).
\end{proof}

We now introduce some more notation, in order to state the ``Reversed Brunn-Minkowski inequality'' (page 3 of \cite{KT}, Theorem 2.4 \cite{KK1}).
 Let $C$ be a closed, strictly convex cone in $\RR^d$ with apex at the origin. A closed convex set $D\subset C$ is $C$-convex if for any $x\in D$ we have that $x+C\subset D$. $D$ is cobounded if $C\setminus D$ is bounded. Define
 $$
 {\rm covol}(D)={\rm vol}(C\setminus D).
 $$
 
 \begin{Theorem}\label{TheoremM21}(Khovanskii and Timorin) Let $D_1$ and $D_2$ be cobounded $C$-convex regions in a cone $C$. Then 
 \begin{equation}\label{eqM6}
 {\rm covol}^{\frac{1}{d}}(D_1)+{\rm covol}^{\frac{1}{d}}(D_2)\ge{\rm covol}^{\frac{1}{d}}(D_1+D_2).
 \end{equation}
 \end{Theorem}
 
 Define
 $$
 \Gamma(I_*)=\{(n_1,\ldots,n_d,i)\in \NN^{d+1}\mid I_i\cap K_{n_1\lambda_1+\cdots+n_d\lambda_d}/I_i\cap K_{n_1\lambda_1+\cdots+m_d\lambda_d}^+\ne 0\}
 $$ 
 and
 $$
 \Gamma(R)=\{(n_1,\ldots,n_d,i)\in \NN^{d+1}\mid R\cap K_{n_1\lambda_1+\cdots+n_d\lambda_d}/R\cap K_{n_1\lambda_1+\cdots+m_d\lambda_d}^+\ne 0\}.
 $$
 Forgetting the 1 in the $(d+1)^{\rm st}$ coefficient, we can regard $\Delta(\Gamma(R))$ and $\Delta(\Gamma(I_*))$ as subsets of $\RR^d$.

$\Delta(\Gamma(R))$ is a strongly convex closed $d$-dimensional cone in $\RR^d$ (since $R$ is a ring). Further, $\Delta(\Gamma(I_*))\subset \Delta(\Gamma(R))$ is a closed convex subset which is $\Delta(\Gamma(R))$-convex (since the $I_i$ are ideals in $R$. 
Further $\Delta(\Gamma(I_*))$ is cobounded by (\ref{eqF30}).
We have that
$$
{\rm covol}(\Delta(\Gamma(I_*))={\rm vol}(\Delta(\Gamma(R))\setminus \Delta(\Gamma(I_*)))={\rm vol}(\Delta(\hat\Gamma^{(1)}))    -{\rm vol}(\Delta(\Gamma^{(1)}))
$$
by (\ref{eqF30}). Theorem \ref{TheoremM5} now becomes
\begin{equation}\label{eqM30}
\lim_{i\rightarrow\infty}\frac{\ell_R(R/I_i)}{i^d}=[k':k]{\rm covol}(\Delta(\Gamma(I_*)).
\end{equation}

We now give the proof of Theorem \ref{TheoremM7}.

We first prove the theorem in the case that $R$ is analytically irreducible. We have  that the Minkowski sum
$$
\Delta(\Gamma(I_*))+\Delta(\Gamma(J_*))\subset \Delta(\Gamma(K_*)).
$$
This follows since if $(m_1,\ldots,m_d,i)\in \Gamma(I_*)$ and $(n_1,\ldots,n_d,j)\in \Gamma(J_*)$, then
$$
(jm_1+in_1,\ldots,jm_d+in_d,ij)\in \Gamma(K_*),
$$
so
$$
\frac{1}{i}(m_1,\ldots,m_d)+\frac{1}{j}(n_1,\ldots,n_d)\in \Delta(\Gamma(K_*)).
$$
Thus
\begin{equation}\label{eqM8}
{\rm covol}^{\frac{1}{d}}(\Delta(\Gamma(I_*)))+{\rm covol}^{\frac{1}{d}}(\Delta(\Gamma(J_*)))
\ge{\rm covol}^{\frac{1}{d}}(\Delta(\Gamma(I_*))+\Delta(\Gamma(J_*)))
\ge {\rm covol}^{\frac{1}{d}}(\Delta(\Gamma(K_*)))
\end{equation}
by (\ref{eqM6}). The theorem now follows, in the case that $R$ is analytically irreducible, from (\ref{eqM30}).

Now suppose that $R$ is an arbitrary local ring of dimension $d$ with $\dim N(\hat R)<d$. Let $A=\hat R/N(\hat R)$ and let $P_i$, for $1\le i\le t$ be the minimal prime ideals of $A$. Let $A_i=A/P_i$. Each $A_i$ is an analytically irreducible local ring of dimension $d$. As in the first part of the proof of Theorem \ref{Theorem1}, using Lemma \ref{LemmaM30}, we have that
\begin{equation}\label{eqM7}
\lim_{n\rightarrow \infty}\frac{\ell_R(R/I_n)}{n^d}=\lim_{n\rightarrow\infty}\frac{\ell_R(A/I_nA)}{n^d}=\sum_{i=1}^t
\lim_{n\rightarrow\infty}\frac{\ell_R(A_i/I_nA_i)}{n^d}.
\end{equation}
For $1\le i\le t$, let
$$
a_i=\lim_{n\rightarrow\infty}\frac{\ell_R(A_i/I_nA_i)}{n^d},
b_i=\lim_{n\rightarrow\infty}\frac{\ell_R(A_i/J_nA_i)}{n^d},
c_i=\lim_{n\rightarrow\infty}\frac{\ell_R(A_i/K_nA_i)}{n^d}.
$$
By Theorem \ref{TheoremM7} in the case that $R$ is analytically irreducible, we have that
$$
a_i^{\frac{1}{d}}+b_i^{\frac{1}{d}}\ge c_i^{\frac{1}{d}}\mbox{ for }1\le i\le t.
$$
Setting 
$$
\overline a_i=a_i^{\frac{1}{d}},\,\,\overline b_i=b_i^{\frac{1}{d}}\mbox{ for }1\le i\le t,
$$
we have by Minkowski's inequality (Formula (2.11.4) on page 31 of \cite{HLP})
$$
\begin{array}{lll}
(\sum_{i=1}^t a_i)^{\frac{1}{d}}+(\sum_{i=1}^t b_i)^{\frac{1}{d}}
&=&(\sum_{i=1}^t \overline a_i^d)^{\frac{1}{d}}+(\sum_{i=1}^t \overline b_i^d)^{\frac{1}{d}}\\
&\ge&(\sum_{i=1}^t(\overline a_i+\overline b_i)^d)^{\frac{1}{d}}\\
&\ge&(\sum_{i=1}^tc_i)^{\frac{1}{d}}.
\end{array}
$$
By (\ref{eqM7}), we have established the conclusions of Theorem \ref{TheoremM7}.

 \section{Proof of Theorem \ref{TheoremC}}\label{Section7}
 
In this section we give the proof of Theorem \ref{TheoremC}.

\subsection{More cones associated to semigroups}\label{SecCone}
We first summarize some results on semigroups and associated cones from \cite{KK}, which generalize Theorem \ref{Theorem3} stated in Section \ref{Section5}.

Suppose that $S$ is a subsemigroup of $\ZZ^{d}\times \NN$ which is not contained in $\ZZ^d\times\{0\}$. Let $L(S)$ be the subspace of $\RR^{d+1}$ which is generated by $S$, and let $M(S)=L(S)\cap(\RR^d\times\RR_{\ge 0})$. 

Let $\mbox{Con}(S)\subset L(S)$ be the closed convex cone which is the closure of  the set of all linear combinations $\sum \lambda_is_i$ with $s_i\in S$ and $\lambda_i\ge 0$.

$S$ is called {\it strongly nonnegative} (Section 1.4 \cite{KK}) if $\mbox{Cone}(S)$  intersects $\partial M(S)$ only at the origin (this is equivalent to being strongly admissible (Definition 1.9 \cite{KK}) since with our assumptions, $\mbox{Cone}(S)$ is contained in $\RR^d\times\RR_{\ge 0}$,  so the ridge of of $S$ must be contained in $\partial M(S)$). In particular, a subsemigroup of a strongly negative semigroup is itself strongly negative.

We now introduce some notation from \cite{KK}. Let 
\vskip .1truein

$S_k=S\cap (\RR^d\times\{k\})$.

$\Delta(S)=\mbox{Con}(S)\cap (\RR^{d}\times\{1\})$ (the Newton-Okounkov body of $S$).

$q(S)=\dim \partial M(S)$.

$G(S)$ be the subgroup of $\ZZ^{d+1}$ generated by $S$.

$m(S)=[\ZZ:\pi(G(S))]$
  where $\pi:\RR^{d+1}\rightarrow \RR$ be projection onto the last factor.

$\mbox{ind}(S)= [\partial M(S)_{\ZZ}:G(S)\cap \partial M(S)_{\ZZ}]$
where 

$\partial M(S)_{\ZZ}:=\partial M(S)\cap \ZZ^{d+1}= M(S)\cap (\ZZ^d\times\{0\})$.

${\rm vol}_{q(S)}(\Delta(S))$ is the integral volume of $\Delta(S)$. This volume is computed using the translation of the integral measure on $\partial M(S)$.
\vskip .2truein

$S$ is strongly negative if and only if $\Delta(S)$ is a compact set. If $S$ is strongly negative, then the dimension of $\Delta(S)$ is $q(S)$.
\vskip .1truein

\begin{Theorem}\label{ConeTheorem3}(Kaveh and Khovanskii) Suppose that $S$ is strongly nonnegative.  Then 
$$
\lim_{k\rightarrow \infty}\frac{\#S_{m(S)k}}{k^{q(S)}}=\frac{{\rm vol}_{q(S)}(\Delta(S))}{{\rm ind}(S)}.
$$
\end{Theorem}

This is proven in  Corollary 1.16 \cite{KK}. 

With our assumptions, we have that $S_n=\emptyset$ if $m(S)\not\,\mid n$ and  the limit is positive, since
${\rm vol}_{q(S)}(\Delta(S))>0$.

\begin{Theorem}\label{ConeTheorem4} Suppose that $q$ is a positive integer such there exists a sequence $k_i\rightarrow \infty$ of positive integers such that  the sequence $\#S_{m(S)k_i}/k_i^q$ is bounded. Then $S$ is strongly nonnegative with $q(S)\le q$.
\end{Theorem}

This is proven in Theorem 1.18 \cite{KK}.

\subsection{Limits for graded algebras over a local domain}

\begin{Theorem}\label{TheoremA} Suppose that $R$ is an analytically irreducible local domain, 
$$
B=R[x_1,\ldots,x_n]=\bigoplus_{k\ge 0}B^k
$$
is a standard graded polynomial ring over $R$ and $A=\bigoplus_{k\ge 0}A^k$ is a graded $R$-subalgebra of $B$. Suppose that $A^1\ne 0$ and that $q\in \ZZ_{>0}$ is such that for all $c\in \ZZ_{>0}$, there exists $\gamma_c\in\RR_{>0}$ such that
\begin{equation}\label{eq1}
\ell_R(A^k/(m_R^{ck}B)\cap A^k)<\gamma_ck^q
\end{equation}
for all $k\ge 0$. Then for any fixed positive integer $c$,
$$
\lim_{k\rightarrow\infty}\frac{\ell_R(A^k/(m_R^{ck}B)\cap A^k)}{k^q}
$$
exists.
\end{Theorem}

Let $c>0$ be a fixed positive integer.
We first reduce to the case that $R$ is complete. We can do this since $\hat R$ is a flat $R$-module. To begin with,
$$
\hat A:=A\otimes_R\hat R\cong \bigoplus_{k\ge 0}\hat A^k\subset \hat B:=B\otimes_R\hat R\cong \hat R[x_1,\ldots,x_n].
$$
 Tensoring the exact sequence
$$
0\rightarrow (m_R^{ck}B)\cap A^k\rightarrow A^k\rightarrow A^k/(m_R^{ck}B)\cap A^k\rightarrow 0
$$
with $\hat R$ over $R$, and using the fact that $m_R^{ck}A^k\subset (m_R^{ck}B)\cap A^k$, we have exact sequences
$$
0\rightarrow (m_{\hat R}^{ck}\hat B)\cap \hat A^k\cong ((m_R^{ck}B)\cap A^k)\otimes_R\hat R
\rightarrow \hat A^k\rightarrow (A^k/(m_R^{ck}B)\cap A^k)\otimes_R\hat R\cong A^k/(m_R^{ck}B)\cap A^k\rightarrow 0
$$
so that
$$
\ell_{\hat R}(\hat A^k/(m_{\hat R}^{ck}\hat B)\cap \hat A^k)=\ell_R(A^k/(m_R^{ck}B)\cap A^k)
$$
for all $c,k$.

For the duration of the proof we will assume that $R$ is complete.
Let $X$ be the normalization of the blow up $\pi:X\rightarrow \mbox{Spec}(R)$ of  the maximal ideal $m_R$ of $R$. $X$ is of finite type over $R$ since $R$ is excellent
(as it is complete). As $X$ is normal, it is regular in codimension 1, so
there exists a closed point $p\in \pi^{-1}(m_R)$ such that $S=\mathcal O_{X,p}$ is regular. $S$ necessarily  dominates $R$. $S$ is essentially of finite type over $R$ and has the same quotient field $Q(R)$. Let $\ell=[S/m_S:R/m_R]<\infty$.

We first define a valuation $\nu$ dominating $S$ by the method of the proof of Theorem \ref{Theorem1}.
Let $y_1,\ldots,y_d$ be regular parameters in $S$, and let $\lambda_1,\ldots,\lambda_d$ with $\lambda_i\ge 1$ be rationally independent real numbers. Define a valuation $\nu$ on $Q(R)$ by
$\nu(y_i)=\lambda_i$ for $1\le i\le d$ and $\nu(\gamma)=0$ if $\gamma$ is a unit in $S$. The value group $\Gamma_{\nu}$ of $\nu$ is
the ordered subgroup $\Gamma_{\nu}=\ZZ\lambda_1+\cdots+\ZZ\lambda_d$ of $\RR$, which is isomorphic to $\ZZ^d$ as an unordered group.
Let $V_{\nu}$ be the valuation ring of $\nu$. The residue field of $V_{\nu}$ is $V_{\nu}/m_{\nu}\cong S/m_S$.

We now extend $\nu$ to a valuation $\omega$ on the rational function field $Q(R)(x_1,\ldots,x_n)$ with value group
$\Gamma_{\omega}=(\Gamma_{\nu}\times \ZZ^n)_{\rm lex}$, by defining
$$
\omega(g)=\min\{(\nu(a_{i_1,\ldots,i_n}),i_1,\ldots,i_n)\mid a_{i_1,\ldots,i_n}\ne 0\}
$$
for  $g=\sum a_{i_1,\ldots,i_n}x_1^{i_1}\cdots x_n^{i_n}\in Q(R)[x_1,\ldots,x_n]$ with $a_{i_1,\ldots,i_n}\in Q(R)$.
We have that $V_{\omega}/m_{\omega}\cong V_{\nu}/m_{\nu}\cong S/m_S$.

Define valuation ideals $K(\nu)_{\lambda}$ and $K(\nu)_{\lambda}^+$ in $V_{\nu}$ for $\lambda\in \Gamma_{\nu}$ and 
$K(\omega)_{\tau}$ and $K(\omega)_{\tau}^+$ in $V_{\omega}$ for $\tau\in \Gamma_{\omega}$ to be the respective sets of elements of
$\nu$-value $\ge \lambda$, $\nu$-value $> \lambda$, $\omega$-value $\ge \tau$ and $\omega$-value $> \tau$. We have, as in Lemma \ref{Prop1},

\begin{Lemma}\label{LemmaF1}(Lemma 4.3 \cite{C4}) There exists $\beta\in \ZZ_{>0}$ such that $K(\nu)_{\beta k}\cap R\subset m_R^{ck}$
for all $k\in \NN$.
\end{Lemma}

We conclude that
$$
K(\omega)_{k(\beta,0,\ldots,0)}\cap B\subset m_R^{kc}B
$$
for all $k$, so that
$$
K(\omega)_{k(\beta,0,\ldots,0)}\cap A^k=  K(\omega)_{k(\beta,0,\ldots,0)}\cap ((m_R^{kc}B)\cap A^k) \subset (m_R^{kc}B)\cap A^k
$$
for all $k$.

Let $\overline A^k=(m_R^{kc}B)\cap A^k$. We have that
\begin{equation}\label{eqG60}
\ell_R(A^k/(m_R^{ck}B)\cap A^k)
=\ell_R(A^k/K(\omega)_{k(\beta,0,\ldots,0)}\cap A^k)
-\ell_R(\overline A^k/K(\omega)_{k(\beta,0,\ldots,0)}\cap \overline A^k)
\end{equation}
for all $k$.

For $t\ge 1$ define

$$
\Gamma^{(t)} = \left\{\begin{array}{lll}
(n_1,\ldots,n_d,i_1,\ldots,i_n,k)\in \NN^{d+n+1}&\mid&
\dim_{R/m_R}\frac{A^k\cap K(\omega)_{(n_1\lambda_1+\cdots+n_d\lambda_d,i_1,\ldots,i_n)}}
{A^k\cap K(\omega)^+_{(n_1\lambda_1+\cdots+n_d\lambda_d,i_1,\ldots,i_n)}}\ge t\\
&&\mbox{ and }n_1+\cdots+n_d\le \beta k
\end{array}\right\}
$$
and 
$$
\overline \Gamma^{(t)} = \left\{\begin{array}{lll}
(n_1,\ldots,n_d,i_1,\ldots,i_n,k)\in \NN^{d+n+1}&\mid&
\dim_{R/m_R}\frac{\overline A^k\cap K(\omega)_{(n_1\lambda_1+\cdots+n_d\lambda_d,i_1,\ldots,i_n)}}
{\overline A^k\cap K(\omega)^+_{(n_1\lambda_1+\cdots+n_d\lambda_d,i_1,\ldots,i_n)}}\ge t\\
&&\mbox{ and }n_1+\cdots+n_d\le \beta k
\end{array}\right\}.
$$
For all $k$ and $\tau$ we have natural $R/m_R$-vector space inclusions
$$
A^k\cap K(\omega)_{\tau}/A^k\cap K(\omega)^+_{\tau}\rightarrow V_{\omega}/m_{\omega}
$$
and
$$
\overline A^k\cap K(\omega)_{\tau}/\overline A^k\cap K(\omega)^+_{\tau}\rightarrow V_{\omega}/m_{\omega}
$$
so $\Gamma^{(t)}=\emptyset$ for $t>\ell$ and $\overline \Gamma^{(t)}=\emptyset$ for $t>\ell$.
We have that
$$
\ell_R(K(\omega)_{\lambda}\cap A^k/K(\omega)^+_{\lambda}\cap A^k)
=\#\{t\mid (n_1,\ldots,n_d,i_1,\ldots,i_n,k)\in \Gamma^{(t)}\}
$$
for $\lambda=(n_1\lambda_1+\cdots+n_d\lambda_d,i_1,\ldots,i_n)$ such that $n_1+\cdots+n_d\le\beta k$, and the corresponding statement for $\overline \Gamma^{(t)}$ also holds.  We have that 
$$
\lambda=n_1\lambda_1+\cdots+n_d\lambda_d<k(\beta,0,\ldots,0)
$$
(in the lex order) if and only if $n_1\lambda_1+\cdots+n_d\lambda_d<k\beta$.
Since $\lambda_i\ge 1$ for all $i$, this implies $n_1+\cdots+n_d\le\beta k$. Thus
\begin{equation}\label{eqF51}
\ell_R(A^k/ K(\omega)_{k(\beta,0,\ldots,0)}\cap A^k)
=\sum_{0\le \lambda<k(\beta,0,\ldots,0)}(\ell_{R/m_R} K(\omega)_{\lambda}\cap A^k/K(\omega)^+_{\lambda}\cap A^k)
=\sum_{t=1}^{\ell}\#\Gamma^{(t)}_k
\end{equation}
and
\begin{equation}\label{eqF52}
\ell_R(\overline A^k/ K(\omega)_{k(\beta,0,\ldots,0)}\cap \overline A^k)
=\sum_{0\le \lambda<k(\beta,0,\ldots,0)}(\ell_{R/m_R} K(\omega)_{\lambda}\cap \overline A^k/K(\omega)^+_{\lambda}\cap \overline A^k)
=\sum_{t=1}^{\ell}\#\overline \Gamma^{(t)}_k.
\end{equation}

The proof of the following Lemma \ref{Lemma2} is similar to the proof of Lemma \ref{LemmaM10}.

\begin{Lemma}\label{Lemma2} 
Suppose that $t\ge 1$, $0\ne f\in A^i$, $0\ne g\in A^j$ and 
$$
\ell_{R/m_R} \left(A^i\cap K(\omega)_{\omega(f)}/A^i\cap K(\omega)^+_{\omega(f)}\right)\ge t.
$$
Then 
$$
\ell_{R/m_R} \left(A^{i+j}\cap K(\omega)_{\omega(fg)}/A^{i+j}\cap K(\omega)^+_{\omega(fg)^+}\right)\ge t.
$$
Suppose that $t\ge 1$, $0\ne f\in \overline A^i$, $0\ne g\in \overline A^j$ and 
$$
\ell_{R/m_R} \left(\overline A^i\cap K(\omega)_{\omega(f)}/\overline A^i\cap K(\omega)^+_{\omega(f)}\right)\ge t.
$$
Then 
$$
\ell_{R/m_R} \left(\overline A^{i+j}\cap K(\omega)_{\omega(fg)}/\overline A^{i+j}\cap K(\omega)^+_{\omega(fg)^+}\right)\ge t.
$$
\end{Lemma}

\begin{Proposition}\label{PropF1} Suppose that $\Gamma^{(t)}\not\subset \{0\}$.
Then 
\begin{enumerate}
\item[1)] $\Gamma^{(t)}$ is a subsemigroup of $\NN^{d+n+1}$.
\item[2)] $\Gamma^{(t)}$ is strongly nonnegative with $q(\Gamma^{(t)})\le q$.
\item[3)] $m(\Gamma^{(t)})=1$.
\item[4)] $\overline \Gamma^{(t)}$ is a subsemigroup of $\NN^{d+n+1}$.
\item[5)] $\overline \Gamma^{(t)}$ is strongly nonnegative with $q(\overline \Gamma^{(t)})\le q$.
\item[6)] $m(\overline \Gamma^{(t)})=1$.
\end{enumerate}
\end{Proposition}

\begin{proof}  We will prove the Proposition for $\Gamma^{(t)}$. The proof for $\overline\Gamma^{(t)}$ is the same.
It follows from Lemma \ref{Lemma2} that $\Gamma^{(t)}$ is a subsemigroup of $\NN^{d+n+1}$.

$m_R^{k\beta}\subset K(\nu)_{k\beta}\cap R$ for all $k$ (since $\lambda_i\ge 1$ for all $i$), so 
$$
(m_R^{kc\beta}B)\cap A^k\subset K(\omega)_{k(\beta,0,\ldots,0)}\cap A^k.
$$
Thus
$$
\#\Gamma_k^{(t)}\le \ell_R(A^k/K(\omega)_{k(\beta,0,\ldots,0)}\cap A^k)
\le \ell_R(A^k/(m_R^{kc\beta}B)\cap A^k)\le \gamma_{c\beta}k^q
$$
for all $k$ by (\ref{eq1}). By Theorem \ref{ConeTheorem4}, $\Gamma^{(t)}$ is thus strongly nonnegative and $q(\Gamma^{(t)})\le q$. 

By assumption $\Gamma_i^{(t)}\neq\emptyset$ for some $i\ge 1$. Thus there exists $0\ne f\in A^i$ such that 
$$
\omega(f)=n_1\lambda_1+\cdots+n_d\lambda_d+i_1+\cdots+i_n
$$
with $n_1+\cdots+n_d\le \beta i$ and
$$
\ell_{R/m_R}\left(A^i\cap K(\omega)_{\omega(f)}/A^i\cap K(\omega)_{\omega(f)}^+\right)\ge t.
$$
By assumption, there exists $0\ne g\in A^1$. Let 
$$
\omega(g)=m_1\lambda_1+\cdots+m_d\lambda_d+j_1+\cdots+j_n.
$$
After  increasing $\beta$ if necessary, we may assume that $m_1+\cdots+m_d\le\beta j$. Thus
$$
\omega(fg)=(m_1+n_1)\lambda_1+\cdots+(m_d+n_d)\lambda_d+(i_1+j_1)+\cdots+(i_n+j_n)
$$
with $(m_1+n_1)+\cdots+(m_d+n_d)\le\beta(i+j)$. Thus
$\Gamma^{(t)}_{k+1}\ne \emptyset$ by Lemma \ref{Lemma2}, so that $m(\Gamma^{(t)})=1$.

\end{proof}

It thus follows from Theorem \ref{ConeTheorem3} that the limits
$$
\lim_{k\rightarrow \infty} \frac{\#\Gamma_k^{(t)}}{k^{q}}\mbox{ and }\lim_{k\rightarrow \infty}\frac{\#\overline \Gamma_k^{(t)}}{k^q}
$$
exist. The conclusions of Theorem \ref{TheoremA} now follow from (\ref{eqG60}), (\ref{eqF51}) and (\ref{eqF52}).

\subsection{Limits for graded algebras over a reduced local ring}

\begin{Theorem}\label{TheoremB} Suppose that $R$ is an analytically unramified local ring, 
$$
B=R[x_1,\ldots,x_n]=\bigoplus_{k\ge 0}B^k
$$
is a standard graded polynomial ring over $R$ and $A=\bigoplus_{k\ge 0}A^k$ is a graded $R$-subalgebra of $B$. Suppose that 
if $P$ is a minimal prime of $R$ and $A^1/PB\cap A^1=0$ then $A^k/PB\cap A^k=0$ for all $k\ge 1$.
Further suppose that $q\in \ZZ_{>0}$ is such that for all $c\in \ZZ_{>0}$, there exists $\gamma_c\in\RR_{>0}$ such that
\begin{equation}\label{eq2}
\ell_R(A^k/(m_R^{ck}B)\cap A^k)<\gamma_ck^q
\end{equation}
for all $k\ge 0$. Then for any fixed positive integer $c$,
$$
\lim_{k\rightarrow\infty}\frac{\ell_R(A^k/(m_R^{ck}B)\cap A^k)}{k^q}
$$
exists.
\end{Theorem}

Let $c>0$ be a fixed positive integer. 
Let $R_i=R/P_i$  and $C_i=B\otimes_RR/P_i\cong R/P_i[x_1,\ldots,x_n]$ for $1\le i\le s$. As a graded ring, $C_i=\bigoplus C_i^k$
where $C_i^k\cong B^k\otimes_RR/P_i$ as free $R/P_i$-modules. Let $C=\bigoplus_{i=1}^s C_i$. Let $\phi:B\rightarrow C$ be the natural homomorphism. $\phi$ is 1-1 since its kernel is $\cap P_iB= (\cap P_i)B=0$. By Artin-Rees, there exists a positive integer $\lambda$ such that
\begin{equation}\label{eqF40}
\omega_n:=\phi^{-1}(m_R^nC)=B\cap m_R^nC\subset m_R^{n-\lambda}B
\end{equation}
for all $n\ge \lambda$.
Thus
\begin{equation}\label{eqF41}
m_R^nB\subset \omega_n\subset m_R^{n-\lambda}B
\end{equation}
for all $n\ge \lambda$. We have that
$$
\omega_n=\phi^{-1}(m_R^nC)=\phi^{-1}(m_R^nC_1\bigoplus\cdots\bigoplus m_R^nC_s)
=\left[(m_R^n+P_1)B\right]\cap\cdots\cap \left[(m_R^n+P_s)B\right].
$$
Let $\beta=(\lambda+1)c$. We have that 
$$
\omega_{\beta n}\subset m_R^{c(\lambda+1)n-\lambda}B\subset m_R^{cn}B
$$
for all $n\ge 1$. Thus
\begin{equation}\label{eq42}
\ell_R(A^n/((m_R^{cn}B)\cap A^n))=
\ell_R(A^n/(\omega_{\beta n}\cap A^n))-\ell_R((m_R^{cn}B)\cap A^n)/(\omega_{\beta n}\cap A^n))
\end{equation}
for all $n\ge 1$.

Define $L_0^j=R$ for $0\le j\le s$, and for $n>0$, define $L_n^0=A^n$ and
$$
L_n^j=\left[(m_R^{\beta n}+P_1)B\right]\cap\cdots\cap \left[(m_R^{\beta n}+P_j)B\right]\cap A^n.
$$
Let $L^j=\bigoplus_{n\ge 0}L_n^j$, a graded $R$-subalgebra of $B$.
For $0\le j\le s-1$ and $n\ge 1$, we have isomorphisms of $R$-modules
$$
\begin{array}{lll}
L_n^j/L_n^{j+1}&\cong& L_n^j/\left[(m_R^{\beta n}+P_{j+1})B\cap L_n^j\right]\\
&\cong& (L_n^j/P_{j+1}B^n)/\left((L_n^j/P_{j+1}B^n)\cap m_R^{\beta n}(B^n/P_{j+1}B^n)\right)\\
&\cong& \left[L_n^jC_{j+1}^n\right]/\left([L_n^jC_{j+1}^n]\cap[m_R^{\beta n}C_{j+1}]\right)
\end{array}
$$
and 
$$
L_n^s\cong \omega_{\beta n}\cap A^n.
$$
Thus
\begin{equation}\label{eq43}
\begin{array}{lll}
\ell_R(A^n/\omega_{\beta n}\cap A^n)&=& \sum_{j=0}^{s-1}\ell_R(L_n^j/L_n^{j+1})\\
&=& \sum_{j=0}^{s-1}\ell_{R_{j+1}}\left(
L_n^jC_{j+1}^n/\left[(L_n^jC_{j+1}^n)\cap m_R^{\beta n}C_{j+1}\right]\right).
\end{array}
\end{equation}
For some fixed $j$ with $0\le j\le s-1$, let
$$
\overline R=R/P_{j+1},\,\,\,\overline C=C_{j+1},\,\,\,\overline A^n=L_n^j\overline C^n.
$$
By assumption, If $\overline A^1=0$ then $\overline A^n=0$ for all $n\ge 1$, so we may assume that $\overline A^1\ne 0$.
Since $\overline R$ is analytically irreducible, by Theorem \ref{TheoremA},
$$
\lim_{n\rightarrow \infty}\frac{\ell_{\overline R}(\overline A^n/(m_R^{\beta n}\overline C)\cap \overline A^n)}{n^q}
$$
exists. Thus
$$
\lim_{n\rightarrow \infty}\frac{\ell_R(A^n/\omega_{\beta n}\cap A^n)}{n^q}
$$
exists by (\ref{eq43}). The same argument (from (\ref{eq42})) applied to $(m_R^{cn}B)\cap A^n$ (instead of $A^n$) implies
$$
\lim_{n\rightarrow \infty}\frac{\ell_R([(m_R^{cn}B)\cap A^n]/[\omega_{\beta n}\cap A^n])}{n^q}
$$
exists, so
$$
\lim_{n\rightarrow \infty}\frac{\ell_R(A^n/[(m_R^{cn}B)\cap A^n])}{n^q}
$$
exists by (\ref{eq42}).

\subsection{The proof of Theorem \ref{TheoremC}}

In this subsection, we prove Theorem \ref{TheoremC}. We assume throughout this section that $R$, $E$ and $F$ satisfy the assumptions of Theorem \ref{TheoremC}.
In particular, we suppose that $R$ is a $d$-dimensional,  analytically unramified local ring, and $E$ is a rank $e$
submodule of a free (finite rank) $R$-module $F=R^n$.
Let $B=R[F]$ be the symmetric algebra of $F$ over $R$, which is isomorphic to a standard graded polynomial ring $B=R[x_1,\ldots,x_n]=\bigoplus_{k\ge 0}F^k$  over $R$. We may identify $E$ with a submodule $E^1$ of $B^1$, and let $R[E]=\bigoplus_{n\ge 0}E^k$ be the $R$-subalgebra of $B$ generated by $E^1$ over $R$.

Let $P_1,\ldots, P_s$ be the minimal primes of $R$. Since $R$ is reduced,   $\cap P_i=(0)$ and the total quotient field of $R$ is isomorphic to 
$\bigoplus_{i=1}^sR_{P_i}$, with $R_{P_i}\cong (R/P_i)_{P_i}$. The assumption that $\mbox{rank}(E)=e$ is simply that $E\otimes_RR_{P_i}$ has rank $e$ for all $i$.

\begin{Lemma}\label{Lemma30} The Krull dimension of $R[E]$ is $d+e$.
\end{Lemma}

\begin{proof}
Let $B_i=B/P_iB\cong R/P_i[x_1,\ldots,x_n]$. $\cap (P_iB)= (\cap P_i)B=0$ so the natural homomorphism $B\rightarrow \bigoplus_{i=1}^s B_i$ is 1-1.
Let $E_i=E(F\otimes_R(R/P_i))$. Let $A_i=(R/P_i)[E_i]$ be the graded $R/P_i$-subalgebra of $B_i$ generated by $E_i$. $A_i$ is the image of the natural graded homomorphism from $R[E]$ into $B_i$. Let $K_i$ be the kernel of $R[E]\rightarrow A_i$. The natural homomorphism
$R[E]\rightarrow \bigoplus_{i=1}^sA_i$ is 1-1 since this map factors through the   composition
$R[E]\rightarrow B\rightarrow \bigoplus_{i=1}^sB_i$ of 1-1 homomorphisms. Thus $\cap K_i=(0)$.

$\mbox{rank}(E)=e$ implies $\mbox{rank}(E\otimes_R R_{P_i})=e$ for all $i$ so that $\mbox{rank}(E_i)=e$ for all $i$, since $E_i\otimes_RR_{P_i}\cong E\otimes_RR_{P_i}$. Thus 
$A_i\otimes_{R/P_i}(R/P_i)_{P_i}$ is an $e$-dimensional polynomial ring over the field $R_{P_i}\cong (R/P_i)_{P_i}$. Thus $\mbox{trdeg}_{R/P_i}A_i=e$. Since $R/P_i$ is a Noetherian domain and $(R/P_i)[E_i]$ is a finitely generated $R/P_i$-algebra,
$$
\dim A_i=\dim R/P_i+\mbox{\rm trdeg}_{R/P_i}A_i
$$
by Lemma 1.2.2 \cite{V}.

Since any prime ideal in $R[E]$ must contain some $K_i$, we obtain from the definition of Krull dimension that
$$
\dim R[E] = \max\dim A_i=\max\{\dim(R/P_i)+e\}=d+e.
$$
\end{proof}

\begin{Lemma}\label{Lemma31}
Suppose that $c$ is a positive integer. Then 
there exists a constant $\beta_c$ such that 
\begin{equation}\label{eqF10}
\ell_R(E^k/(m_R^{ck}B)\cap E^k)<\beta_ck^{d+e-1}
\end{equation}
for all positive integers $k$.
\end{Lemma}

\begin{proof} Let 
$$
A=\bigoplus_{i,j\ge 0}\left(m_R^iE^j/m_R^{i+1}E^j\right).
$$
$A$ is a bigraded algebra over the field $\mathfrak k:=R/m_R$. Let $a_1,\ldots,a_m$ be generators of $m_R$ as an $R$-module and
$b_1,\ldots,b_n$ be generators of $E$ as an $R$-module. Let $S=\mathfrak k[x_1,\ldots,x_m;y_1,\ldots,y_n]$ be a polynomial ring.
$S$ is bigraded by $\deg(x_i)=(1,0)$ and $\deg(y_j)=(0,1)$. The surjective $\mathfrak k$-algebra homomorphism
$S\rightarrow A$ defined by 
$$
x_i\mapsto [a_i]\in m_R/m_R^2,\,\,\,y_j\mapsto [b_j]\in E^1/m_RE^1
$$
is bigraded, realizing $A$ as a bigraded $S$-module. $A\cong {\rm gr}_{m_RR[E]}R[E]$, so
$$
\dim_SA=\dim A\le \dim R[E]=d+e.
$$
Most directly from Theorem 2.4 and 2.2 \cite{A}, or as can be deduced by the general result Theorem 8.20 of  \cite{MS}, there exists a positive integer $k_0$ such that $k\ge k_0$ implies
$$
\ell_R(E^k/m_R^{kc}E^k)=\sum_{i=0}^{kc-1}\dim_{\mathfrak k}\left(m_R^iE^k/m_R^{i+1}E^k\right)
$$
is a polynomial in $k$ of degree $\le d+e-1$. 

$m_R^{ck}E^k\subset (m_R^{ck}F^k)\cap E^k$ implies
$$
\ell_R(E^k/(m_R^{ck}F^k)\cap E^k)\le \ell_R(E^k/m_R^{ck}E^k)
$$
from which the conclusions of the lemma follows.
\end{proof}

Now we prove Theorem \ref{TheoremC}.

Let $I=E^1B$, the ideal generated by $E^1$ in
$B$.   By Theorem 3.4 \cite{Sw}. for all $k\ge 1$, there exist irredundant primary decompositions
$$
I^k=q_1(k)\cap \cdots \cap q_t(k)
$$
and a positive integer $c_0$ such that 
$$
\sqrt{q_i(k)}^{ck}\subset q_i(k)
$$
for all $k$. Suppose that $c\ge c_0$. Since 
$$
I^k:_B(m_RB)^{\infty}=\cap_{m_RB\not\subset \sqrt{q_i(k)}}\cap q_i(k),
$$
we have that 
\begin{equation}\label{eq11}
(m_RB)^{ck}\cap (E^1B)^k:_B(m_RB)^{\infty})\subset (E^1B)^k
\end{equation}
for all positive integers $k$. Now
\begin{equation}\label{eqF13}
(E^1B)^k\cap F^k=E^k,\,\,\,(m_RB)^{ck}\cap F^k=m_R^{ck}F^k
\end{equation}
and
\begin{equation}\label{eq12}
\begin{array}{lll}
[(m_RB)^{ck}\cap (E^1B)^k:_B(m_RB)^{\infty}]\cap F^k
&=&[(m_RB)^{ck}\cap F^k]\cap [((E^1B)^k:_B(m_RB)^{\infty})\cap F^k]\\
&=&(m_R^{ck}F^k)\cap (E^k:_{F^k}m_R^{\infty}).
\end{array}
\end{equation}
Thus for all $k$,
$$
m_R^{ck}F^k\cap E^k\subset (m_R^{ck}F^k)\cap (E^k:_{F^k}m_R^{\infty})
\subset (m_R^{ck}F^k)\cap[(E^1B)^k\cap F^k] 
=(m_R^{ck}F^k)\cap E^k.
$$
Hence
\begin{equation}\label{eq14}
(m_R^{ck}F^k)\cap E^k=(m_R^{ck}F^k)\cap (E^k:_{F^k}m_R^{\infty})
\end{equation}
for $c\ge c_0$ and all positive integers $k$.

Now from (\ref{eq14}), for $c\ge c_0$ and all positive integers $k$, we have short exact sequences
of $R$-modules
\begin{equation}\label{eq15}
0\rightarrow E^k/(m_R^{ck}F^k)\cap E^k\rightarrow (E^k:_{F^k}m_R^{\infty})/((m_R^{ck}F^k)\cap (E^k:_{F^k}m_R^{\infty}))
\rightarrow E^k:_{F^k}m_R^{\infty}/E^k\rightarrow 0.
\end{equation}

Since we assume that the epsilon multiplicity is finite, there exists a positive constant $\alpha$ such that
\begin{equation}\label{eq16}
\ell_R(E^k:_{F^k}m_R^{\infty}/E^k)<\alpha k^{d+e-1}
\end{equation}
for all $k>0$.
From (\ref{eqF10}) and (\ref{eq16}), we obtain bounds
\begin{equation}\label{eq17}
\ell_R(E^k:_{F^k}m_R^{\infty}/((m_R^{ck}F^k)\cap (E^k:_{R^k}m_R^{\infty})))<(\alpha+\beta_c)k^{d+e-1}
\end{equation}
for $c\ge c_0$ and all positive integers $k$.

We have that 
$$
E^1/(P_iF^1)\cap E^1\ne 0 \mbox{ and }(E^1:_{F^1}m_R^{\infty})/(P_iF^1)\cap (E^1:_{F^1}m_R^{\infty})\ne 0
$$
since $\mbox{rank}(E^1\otimes_RR_{P_i})=\mbox{rank}(E)=e$ for $1\le i\le s$.
By Theorem \ref{TheoremB}, (\ref{eqF10}), (\ref{eq17}) and (\ref{eq15}) the conclusions of Theorem \ref{TheoremC} hold.

\end{document}